\documentclass[11pt]{amsart}

\usepackage{amsfonts,amssymb,amsmath,amsthm,mathtools}
\usepackage{url}
\usepackage{enumerate}
\usepackage[T1]{fontenc}
\usepackage{mathrsfs}
\usepackage{psfrag}
\usepackage{pst-all}
\usepackage[dvips]{graphicx}
\usepackage{xcolor}
\usepackage{tikz}
 \usepackage{soul}

\usepackage{cases}
\usepackage{caption, subcaption}
\usepackage{enumitem}       
\usepackage[utf8]{inputenc}
\usepackage[margin=3cm]{geometry}
\usepackage{enumerate}
\usepackage{listings}
\usepackage{amsthm, amsfonts, amssymb, amsmath,mathbbol}

\usepackage{graphicx}
\usepackage{xcolor}
\usepackage{color}
\usepackage{framed}
\usepackage{tikz,tikz-cd}
\usepackage{pgf,tikz}
\usepackage{tcolorbox}
\usepackage{scalefnt}
\usepackage{accents}
\usepackage{caption}
\usepackage{makecell}
\usepackage{float}
\usepackage{thmtools}
\usepackage{thm-restate}
\usepackage{makecell}

\usetikzlibrary{patterns, positioning, arrows,decorations.markings}

\definecolor{rose}{rgb}{0.93, 0.23, 0.51}
\definecolor{ref}{rgb}{0.29, 0.59, 0.82}
\definecolor{bovert}{RGB}{200,255,200}
\usepackage[hypertexnames=false,hidelinks]{hyperref}
\hypersetup{colorlinks,linkcolor=black,citecolor=rose,urlcolor={blue!80!black}}

\newtheorem{lem}{Lemma}[section]
\newtheorem{cor}[lem]{Corollary}
\newtheorem{thm}[lem]{Theorem}
\newtheorem{prop}[lem]{Proposition}

\newtheorem{fact}[lem]{Fact}

\usepackage{cleveref}

\declaretheorem[name=Theorem,numberwithin=section]{theorem}

\theoremstyle{definition}
\newtheorem{definition}[lem]{Definition}
\newtheorem*{definition*}{Definition}
\newtheorem{ex}[lem]{Example}

\newtheorem{rmk}[lem]{Remark}
\newtheorem{question}[lem]{Question}
\newtheorem*{question*}{Question}
\newtheorem*{warning*}{Warning to the reader}

\numberwithin{equation}{section}

\newcommand{\K}{\mathrm{k}}

\newcommand{\R}{\mathbf R}

\newcommand{\tauL}{\mathcal{T}}

\newcommand{\RP}{\mathbf{RP}}

\newcommand{\B}{\mathbf B}

\newcommand{\C}{\mathbf C}

\renewcommand{\H}{\mathbf H}

\newcommand{\dhat}{\smash{\hat{d}}}

\newcommand{\Conv}{\operatorname{Conv}}

\newcommand{\Hcal}{\mathcal{H}}
\newcommand{\Conf}{\mathrm{Conf}}

\newcommand{\Span}{\mathrm{Span}}

\newcommand{\PPL}{\mathsf{PPL}}
\newcommand{\Ric}{\mathsf{Ric}}

\renewcommand{\qed}{$\hfill\blacklozenge$}
\renewcommand{\b}{\mathbf{b}}

\renewcommand{\O}{\operatorname{O}}

\newcommand{\Ccal}{\mathcal{C}}

\newcommand{\light}{{\operatorname{light}}}

\newcommand{\hk}{\mathfrak{h}}

\title{Gromov hyperbolic domains in Minkowski space}

\author{Adam Chalumeau}
\date{}

\begin{document}

\begin{abstract} 
    We investigate domains in Minkowski space that are Gromov hyperbolic with respect to a Kobayashi-like metric introduced by Markowitz in the 1980s. For convex, future complete domains, Gromov hyperbolicity is shown to be equivalent to the stable acausality of the boundary. An analogous characterization is obtained for bounded, convex, causally convex domains in terms of the stable acausality of their Geroch–Kronheimer–Penrose causal boundaries.
    Our approach is based on explicit comparisons between the Markowitz metric, the Sormani--Vega null distance and the quasi-hyperbolic metric. We also make use of dynamical arguments similar to those of Benoist and Zimmer in projective and complex geometry. Finally, we compare the Markowitz metric to the Hilbert metric.
\end{abstract}

\maketitle

\section{Introduction}
\sloppy  

\renewcommand{\thetheorem}{\Alph{theorem}}
\setcounter{theorem}{0}

This article studies domains $\Omega$ in Minkowski space $\R^{1,n}$ for which a certain Kobayashi-like metric is hyperbolic in the sense of Gromov \cite{Gromov_hyp_groups}. This metric, denoted $\delta_\Omega$, was introduced and studied by Markowitz \cite{markowitz_1981,markowitz_warped_product} in the context of conformal pseudo-Riemannian manifolds, as an analogue of the Kobayashi--Hilbert metric in projective geometry \cite{kob_projectif}. It was recently used by Galiay and the author \cite{Cha_Gal} to prove a rigidity result for bounded domains in pseudo-Riemannian Minkowski space, and was later further investigated in \cite{thèse,article_dist_markotitz,thèse_blandine}. In the present work, we focus on the metric properties of the space $(\Omega,\delta_\Omega)$, with particular emphasis on the following question:

\begin{question}
    \label{question_intro}
    For which domains $\Omega\subset\R^{1,n}$ is the metric space $(\Omega,\delta_\Omega)$ Gromov hyperbolic?
\end{question}

This question is motivated by the fact that analogous problems are well understood for other Kobayashi-like metrics. In projective geometry, Karlsson and Noskov \cite{karlsson_Noskov} gave necessary and sufficient conditions on the boundary of a convex domain $\Omega$ for the Gromov hyperbolicity of the Hilbert metric $H_\Omega$, generalizing a result of Benoist \cite{Convexes_divisibles_I} for divisible convex domains. A complete characterization of the hyperbolicity of the Hilbert metric was later obtained by Benoist \cite{Benoist_convexes_quasisymetriques}:

\begin{thm}[Benoist]
    Let $\Omega$ be a bounded convex domain in $\R^n$. Then $(\Omega,H_\Omega)$ is Gromov hyperbolic if and only if $\Omega$ is quasi-symmetrically convex.
\end{thm}

In complex geometry, the Gromov hyperbolicity of the Kobayashi metric $K_\Omega$ on a domain $\Omega \subset \C^n$ is also a central topic. Balogh and Bonk \cite{Balogh2000} proved that the Kobayashi metric on bounded strongly pseudoconvex domains in $\C^n$ is Gromov hyperbolic. Gaussier and Seshadri \cite{GaussierSeshadri} then studied the Gromov hyperbolicity of bounded smooth convex domains, and a complete characterization in this setting was given by Zimmer \cite{Zimmer_complex_1}:

\begin{thm}[Zimmer]
    Let $\Omega$ be a bounded convex domain in $\C^n$ with smooth boundary. Then $(\Omega,K_\Omega)$ is Gromov hyperbolic if and only if $\Omega$ has finite type.
\end{thm}
Further characterizations in more general settings were later obtained in \cite{Zimmer_complex_2,Zimmer_non_smooth,Fiacchi_C2}.
The hyperbolicity of another Kobayashi-like metric was also studied by Fiacchi \cite{Fiacchi_kob_euclidien}, who gave both necessary and sufficient conditions for the Gromov hyperbolicity of the minimal metric introduced by Forstnerič and Kalaj \cite{Forsternivc_Kalaj}. More recently, Wang and Zimmer \cite{Wang_Zimmer} proved that the minimal metric on a bounded convex domain $\Omega$ in $\R^n$ with smooth boundary is Gromov hyperbolic if and only if affine planes have finite order of contact with $\partial\Omega$. Their work also provides new proofs of the aforementioned results of Benoist and Zimmer.

The present paper aims to extend the study of Gromov hyperbolicity for Kobayashi-like metrics in the Lorentzian setting.

\subsection{Main results}

We first address Question~\ref{question_intro} in the case where $\Omega$ is a bounded, convex, and \emph{causally convex} domain in $\R^{1,n}$, meaning that any causal curve in $\R^{1,n}$ with endpoints in $\Omega$ is entirely contained in $\Omega$. This assumption is natural from the Lorentzian geometric point of view, as it ensures that $\Omega$, endowed with the induced conformal structure, is globally hyperbolic. In this setting, the boundary of $\Omega$ contains two canonical topological hypersurfaces $\partial_c^+\Omega$ and $\partial_c^-\Omega$, called the \emph{future} and \emph{past causal boundaries}. By the work of Smaï \cite{smaï2023enveloping}, these hypersurfaces are canonically identified with the abstract causal boundary defined by Geroch--Kronheimer--Penrose \cite{Penrose_Kronheimer_Geroch}, see Section~\ref{section_causal_convexité}.

Our first result provides a characterization of the Gromov hyperbolicity of the Markowitz metric in terms of the geometry of the causal boundary. A subset of $\R^{1,n}$ is said to be \emph{stably acausal} if any two of its points are joined by a spacelike segment, even after a small perturbation of the flat metric on $\R^{1,n}$, see Definition~\ref{def_stable_acausalité}.

\begin{theorem}
    \label{thm_intro_CNS_cc}
    Let $\Omega$ be a bounded, convex and causally convex domain in $\R^{1,n}$. Then $(\Omega,\delta_\Omega)$ is Gromov hyperbolic if and only if $\partial_c^+\Omega$ and $\partial_c^-\Omega$ are stably acausal.
\end{theorem}

Somewhat surprisingly, the topological boundary $\partial\Omega$ of the previous domains is never of class $C^1$. This contrasts with the behavior of convex domains in $\RP^n$ for which the Hilbert metric is Gromov hyperbolic: in that setting, the boundary must be $C^1$, see \cite{karlsson_Noskov,Benoist_convexes_quasisymetriques}. In particular, Theorem~\ref{thm_intro_CNS_cc} shows that there exist infinitely many convex domains $\Omega$ in $ \R^{1,n}$ for which the metric $\delta_\Omega$ is not equivalent to the Hilbert metric $H_\Omega$, see Corollary~\ref{corollaire_pour_hilbert}. More generally, these two metrics appear to be non-comparable, see Section~\ref{Section_hilbert} for a complete comparison in dimension two.

The preceding theorem has an analogue for certain unbounded causally convex domains. Recall that a domain $\Omega$ is \emph{future complete} if any future-directed causal curve in $\R^{1,n}$ starting in $\Omega$ is entirely contained in $\Omega$. Convex future complete domains play an important role in the study of globally hyperbolic flat spacetimes \cite{Bonsante,Barbot_flat} and conformally flat spacetimes \cite{Rym_causal_completion,smaï2025futures}, and their geometry is by now well understood, see for instance \cite{seppi_bonsante_Smillie}.

\begin{theorem}
    \label{thm_intro_CNS_futur_complet}
    Let $\Omega$ be a convex future complete domain in $\R^{1,n}$ containing no lightlike line. Then $(\Omega,\delta_\Omega)$ is Gromov hyperbolic if and only if $\partial\Omega$ is stably acausal.
\end{theorem}

Once again, Theorem~\ref{thm_intro_CNS_futur_complet} yields a wide class of Gromov hyperbolic domains $\Omega$ whose boundaries may be non-smooth. For instance, the Einstein--de Sitter space, namely, a half-space bounded by a spacelike hyperplane, is Gromov hyperbolic. In this case, the metric space $(\Omega,\delta_\Omega)$ is bi-Lipschitz isometric to real hyperbolic space $\H^{n+1}$, see Section~\ref{section_examples}. 

In contrast, any convex future complete domain whose boundary contains a lightlike segment is not Gromov hyperbolic by Theorem~\ref{thm_intro_CNS_futur_complet}. For example, regular domains in Minkowski space are not Gromov hyperbolic, since their boundaries contain lightlike segments, see \cite{Bonsante}. 

\begin{rmk}
    Note that in Theorem~\ref{thm_intro_CNS_futur_complet}, the assumption that $\Omega$ contains no lightlike line is used only to ensure that $\delta_\Omega$ defines a distance on $\Omega$, see Proposition~\ref{prop_complétude_delta_Omega}. Moreover, for the ``if'' direction, the convexity assumption on $\Omega$ can be dropped, see Theorem~\ref{thm_général_condi_suffisante}.
\end{rmk}

\subsection{Comparison with the null distance} The present work is an opportunity to compare the Markowitz metric with a recently defined metric on spacetimes, called the \emph{null distance}. This distance was introduced by Sormani--Vega \cite{Sormani_vega} as a metric construction on conformal spacetimes $(M,[g])$ associated to a time function $\tau$ on $M$. It has proven to be a powerful tool in the analysis of spacetimes and their convergence, see for instance \cite{Allen_Burtscher,Kunzinger_Steinbauer, Burtscher_Garcia_Heveling}. Both the Markowitz metric and the null distance are defined in terms of certain notions of length of piecewise lightlike curves, hence it is natural to think that they should be comparable, see our discussion in Section \ref{section_comparaison_generale_Mark_vs_null}. In many cases, we prove that the Markowitz metric is equivalent or quasi-isometrically equivalent to the null distance, see Section \ref{section_comparaison_avec_null_dist}. As explained below, these comparisons are used in the proof of Theorems \ref{thm_intro_CNS_cc} and \ref{thm_intro_CNS_futur_complet}. Another consequence of these comparisons is the following theorem.

\begin{theorem}
    \label{thm_intro_null_dist}
    Theorems \ref{thm_intro_CNS_cc} and \ref{thm_intro_CNS_futur_complet} have analogues for the null distance. Precisely, Theorem \ref{thm_intro_CNS_cc} holds if $\delta_\Omega$ is replaced with the null distance $\dhat_{\ln(\tau^-/\tau^+)}$, where $\tau^+$ and $\tau^-$ are the cosmological time functions of $\Omega$, and Theorem \ref{thm_intro_CNS_futur_complet} holds if $\delta_\Omega$ is replaced with the null distance $\dhat_{\ln(\tau)}$, where $\tau$ is the cosmological time function of $\Omega$.
\end{theorem}

It is worth pointing out that other theorems for the Markowitz metric have analogues for the null distance, see Section \ref{section_consequence_null_distance}.

\subsection{A necessary condition} A consequence of Theorem \ref{thm_intro_CNS_cc} is that bounded, convex and causally convex domains that are Gromov hyperbolic admit nontrivial Cauchy extensions into globally hyperbolic spacetimes. More precisely, these spacetimes are not \emph{$C$-maximal} in the sense of \cite{Clara}, see Section \ref{Section_preuve_non_maximalité}. We are able to prove this fact for non-convex domains.

\begin{theorem}
    \label{thm_intro_maximalité}
    Let $\Omega$ be a bounded, causally convex domain in $\R^{1,n}$. If $(\Omega,\delta_\Omega)$ is Gromov hyperbolic, then $\Omega$ is not $C$-maximal. 
\end{theorem}

The preceding theorem is well illustrated by examples of Theorem \ref{thm_intro_CNS_futur_complet}. Indeed, if $\Omega$ is a convex future complete domain in $\R^{1,n}$ with stably acausal boundary, then the inclusion $\Omega\hookrightarrow\R^{1,n}$ is a non-trivial Cauchy extension, hence $\Omega$ is not $C$-maximal. Theorem \ref{thm_intro_maximalité} shows that if $\Omega$ is a bounded domain in $\R^{1,n}$, and if $(\Omega,\delta_\Omega)$ is Gromov hyperbolic, then $\Omega$ is not \emph{dually convex} in the sense of Zimmer \cite{Zimpropqh}, see Section \ref{Section_preuve_non_maximalité}. This last point was proved by Galiay in her thesis for a broader class of geometric structures called Nagano pairs, under additional assumptions, see \cite[Sec. 8.2.1]{thèse_blandine}. That work extended Zimmer's unpublished work \cite{zimmer_unpublished}. The proof that we develop here is independent of that work and relies on a geometric description of causally convex maximal domains of Smaï \cite{SmaiThese,smaï2025futures}.

\subsection{Strategy} 
The main difficulty for answering Question \ref{question_intro} is the lack of knowledge of Markowitz's geodesics. In general, one can only say that for some domains $\Omega$ in $\R^{1,n}$, lightlike geodesics in $\Omega$ are metric space geodesics of $(\Omega,\delta_\Omega)$, see \cite[Thm. C]{article_dist_markotitz}. 

The proof of the ``only if'' directions of Theorems \ref{thm_intro_CNS_cc} and \ref{thm_intro_CNS_futur_complet} consists in first finding other (quasi)geodesics for the Markowitz metric. This is done using our comparisons with the null distance, which imply that causal curves are $(2,0)$-quasi-geodesics of $(\Omega,\delta_\Omega)$, for any convex causally convex domain $\Omega$ containing no lightlike line, see Lemma \ref{lemme_2_0_quasi_geod_grace_a_la_null_dist}. This leads us to conclude that the boundary of a Gromov hyperbolic domain cannot contain broken lightlike segments, see Theorem \ref{thm_coindi_necessaire_faible}. We then use a dynamical argument, similar to those of Benoist and Zimmer, to show that the boundary of Gromov hyperbolic domains is stably acausal, see Lemma \ref{cor_limite_de_domaines_Gromov_hyp} and Section \ref{section_zooming}.

The proof of the ``if'' directions of Theorems \ref{thm_intro_CNS_cc} and \ref{thm_intro_CNS_futur_complet} is obtained by comparing the Markowitz metric to the \emph{quasi-hyperbolic metric}. 
The Gromov hyperbolicity of this metric is well understood by the work of Bonk--Koskela--Heinonen \cite{Uniformizing_gromov_hyp_space} and Balogh--Buckley \cite{Balogh_Buckley}. The proof consists in showing that for causally convex domains with stably acausal boundary, the Markowitz metric is quantitatively equivalent to the quasi-hyperbolic metric, see Theorem~\ref{thm_equiv_explicite_delta_Omega_et_qh}.

\subsection{Organization of the paper} The Markowitz pseudodistance is introduced in Section \ref{Section_distance_de_mark}. Basic properties of causally convex domains are recalled in Section \ref{section_causal_convexité}. Comparisons with the quasi-hyperbolic distance and the null distance are done in Sections \ref{section_comparaison_avec_dist_quasi_hyp} and \ref{section_comparaison_avec_null_dist}, respectively. The sufficient and necessary conditions of Theorems \ref{thm_intro_CNS_cc} and \ref{thm_intro_CNS_futur_complet} are proved in Sections \ref{section_stab_acaus_implique_Gromov_hyp} and \ref{section_Gromov_hyp_implique_stab_acaus}, respectively. Theorems \ref{thm_intro_null_dist} and \ref{thm_intro_maximalité} are proved in Sections \ref{section_comparaison_avec_null_dist} and \ref{Section_preuve_non_maximalité}, respectively.

\subsection{Acknowledgment} The author is grateful to Andrea Seppi for suggesting comparing the Markowitz metric with the null distance, which ultimately led to this paper. The author is grateful to Pierre-Louis Blayac, Charles Frances, Karin Melnick and Rym Smaï for stimulating discussions during this project.

\section{The pseudodistance}\label{Section_distance_de_mark}

\subsection{Conformal spacetimes} A \emph{conformal Lorentzian manifold} is a manifold $M$ endowed with a conformal class $[g]$ of Lorentzian metrics.
A vector $v\in TM$ is \emph{timelike}, \emph{lightlike}, \emph{spacelike} or \emph{causal} if $g(v,v)$ is negative, null, positive or non-positive, respectively. A \emph{time orientation} on a conformal Lorentzian manifold $(M,[g])$ is a global timelike vector field. Given a time orientation $X$, we say that a causal vector $v$ is \emph{future causal} if $g(v,X)<0$, and  \emph{past causal} otherwise. This infinitesimal language can be translated to curves: a piecewise smooth curve $\gamma:\R\to M$ is a \emph{future causal curve} if its velocity $\gamma^\prime(t)$ is future causal for all $t\in \R$. 
Given two points $x,y$ in a conformal spacetime $(M,[g])$, we write $x\leq y$ (resp. $x\ll y$), if there exists a future causal curve (resp. future timelike curve) from $x$ to $y$.
A topological hypersurface $\Sigma$ in a  conformal spacetime $M$ is a \emph{Cauchy hypersurface} if it intersects every inextensible causal curve exactly once. A conformal spacetime $(M,[g])$ is \emph{globally hyperbolic} if it admits a Cauchy hypersurface.

\subsection{Markowitz's pseudodistance} Let $(M,[g])$ be a conformal spacetime and let $g$ be a metric in the conformal class, with associated Levi-Civita connection $\nabla$. Given a lightlike $\nabla$-geodesic $\gamma(t)$, a function $u(t)$ is a \emph{projective parameter along $\gamma$} if it is a solution to the following differential equation
$$Su(t)=\frac{2}{n-2}\Ric(\dot\gamma(t),\dot\gamma(t)),$$
where $Su$ denotes the Schwarzian derivative of $u$. Given a geodesic $\gamma(t)$ and a projective parameter $u(t)$, the reparametrized geodesic $\gamma(u)=\gamma\circ u^{-1}$ is called a \emph{projectively parametrized lightlike geodesic} of $(M,[g])$. The family of projectively parameterized lightlike geodesics is independent of the choice of metric $g$, and is invariant under precomposition by a homography. In what follows, we denote $I=(-1,1)$ and $\PPL(I,M)$ for the collection of projectively parameterized lightlike geodesics $\gamma:I\to (M,[g])$.

Let $x,y\in M$. A \emph{lightlike chain} from $x$ to $y$, denoted $\Ccal$, consists of a finite sequence $\gamma_1,\dots,\gamma_m\in \PPL(I,M)$ and two finite sequences of points $s_1,\dots,s_m\in I$ and $t_1,\dots,t_m\in I$ such that 
$\gamma_1(s_1)=x$, $\gamma_m(t_m)=y$ and $\gamma_k(t_k)=\gamma_{k+1}(s_{k+1})$ for all $k$.
We will denote $\Ccal=((\gamma_k),(s_k),(t_k))$.

\begin{definition}
    Let $(M,[g])$ be a conformal spacetime. The \emph{Markowitz pseudodistance} between two points $x,y\in M$ is defined as
\begin{equation}
    \label{equation_def_d_mark}
    \delta_M(x,y)=\inf_\Ccal \rho_I(s_1,t_1)+\dots+\rho_I(s_m,t_m),
\end{equation}
where the infimum runs over all lightlike chains $\Ccal=((\gamma_k),(s_k),(t_k))$ from $x$ to $y$, and where $\rho_I$ denotes the distance of $I$ given by $\rho_I(s,t)=\vert\ln\left(\frac{(s-1)(t+1)}{(s+1)(t-1)}\right)\vert$ for all $s,t\in I$. 
\end{definition} 
The Markowitz pseudodistance is symmetric and satisfies the triangle inequality, although it may fail to be definite. It satisfies the following functorial property.
\begin{prop}
    \label{proposition_naturalité}
    Let $f:M\to N$ be a conformal map. Then $\delta_N(f(x),f(y))\leq \delta_M(x,y)$ for all $x,y\in M$. If moreover $f$ is a diffeomorphism, then equality holds in the previous inequality.
\end{prop}

\subsection{The infinitesimal functional} The Markowitz pseudodistance admits an infinitesimal form defined on the bundle of lightlike vectors. 
\begin{definition}
    Let $(M,[g])$ be a conformal spacetime, and let $v\in TM$ be lightlike. The \emph{infinitesimal functional} at $v$ is defined as 
\begin{equation}
    F_M(v)=\inf_\gamma \{\vert u\vert\,\vert\,\gamma\in\PPL(I,M),\,u\in TI \text{ and } \gamma_*u=v\}.
\end{equation}
\end{definition}
The infinitesimal functional and the Markowitz pseudodistance are related by the following formula.

\begin{thm}[{\cite[Thm. 4.8]{markowitz_1981}}]
\label{thm_lien_dmark_Fmark}
    Let $(M,[g])$ be a conformal spacetime. Then, for all $x,y\in M$, one has 
    $$\delta_M(x,y)=\inf_\gamma\int F_M(\gamma^\prime(t)) dt,$$
    where the infimum runs over all piecewise lightlike geodesic curves $\gamma$ from $x$ to $y$.\qed
\end{thm}

\subsection{Projective maps} Let $J$ be a strict interval of $\R$, and let $h$ be a homography such that $h(I)=J$. We let $\rho_J=h_*\rho_I$. This distance on $J$ is independent of the choice of $h$.

\begin{ex}
    If $J=\R_{>0}$, then $\rho_J(x,y)=\vert \ln(x/y)\vert$ for all $x,y\in J$.
\end{ex}

\begin{definition}
    Let $(M,[g])$ be a conformal spacetime. A function $f:M\to \R$ is \emph{projective} if $f\circ \gamma$ is a homography for every $\gamma\in\PPL(I,M)$.
\end{definition}

\begin{prop}[{\cite{markowitz_warped_product}}]
    \label{Proposition_applications_projectives}
    Let $(M,[g])$ be a conformal spacetime, and let $f:M\to \R$ be a projective map. Assume that  $J=f(M)$ is a strict interval. Then, for every $x,y\in M$, one has 
    $$\rho_J(f(x),f(y))\leq \delta_M(x,y).$$
\end{prop}

\subsection{The Minkowski space} The \emph{Minkowski space}, denoted $\R^{1,n}$, is the affine space $\R^{n+1}$ endowed with the canonical flat Lorentzian metric $\b$. Let us fix some notations:

\begin{itemize}
    \item we let $\partial_t$ be a parallel unit timelike vector field on $\R^{1,n}$;
    \item we fix an orthogonal splitting $\R^{n+1}=\R\times\R^n$, where $\R$ is the line generated by $\partial_t$;
    \item we write $x=(t,p)\in \R\times \R^n$ according to the above decomposition. In these coordinates, the metric $\b$ takes the expression:
$$\b=-dt^2+d\sigma^2=-dt^2+dp_1^2+\dots+dp_{n}^2.$$
    \item the parallel timelike vector field $\partial_t$ endows $\R^{1,n}$ with a time orientation. The causal sets associated to a point $x\in \R^{1,n}$ are denoted:
$$I^+(x)=\{y\in \R^{1,n}\,\vert\,x\ll y\},\,\,\,\,\,\,\,\,\,\,\,J^+(x)=\{y\in \R^{1,n}\,\vert\,x\leq y\},\,\,\,\,\,\,\,\,\,\,\,C^+(x)=J^+(x)\setminus I^+(x),$$
$$I^-(x)=\{y\in \R^{1,n}\,\vert\,y\ll x\},\,\,\,\,\,\,\,\,\,\,\,J^-(x)=\{y\in \R^{1,n}\,\vert\,y\leq x\},\,\,\,\,\,\,\,\,\,\,\,\,C^-(x)=J^-(x)\setminus I^-(x).$$
We let $I(x,y)=I^+(x)\cap I^-(y)$ and $J(x,y)=J^+(x)\cap J^-(y)$. Domains of the form $I(x,y)$ are called \emph{diamonds};
\item the \emph{time separation function} $\tauL$ on $\R^{1,n}$ is defined as $\tauL(x,y)=\sqrt{\vert\b(y-x,y-x)\vert}$  for all $x,y\in \R^{1,n}$ such that $x\leq y$;
\item the \emph{Wick rotation} of $\b$ along $\partial_t$ is the Riemannian metric $h=dt^2+d\sigma^2$.
The associated norm and distance on $\R^{1,n}$ are denoted by $\vert \cdot\vert$ and $d_h$, respectively. 
\item For $\varepsilon>0$, we let  $$\b_\varepsilon= -(1+\varepsilon)^2dt^2+dp_1^2+\dots+dp_{n}^2.$$
The conformal spacetime $(\R^{1,n},[\b_\varepsilon])$ is endowed with the same time orientation $\partial_t$, and we denote its causal sets by $I^\pm_\varepsilon(x),J_\varepsilon^\pm(x)$ and $C_\varepsilon^\pm(x)$. 
\end{itemize}

\subsection{The Markowitz metric for domains in Minkowski space} By a \emph{domain} in $\R^{1,n}$, we mean a non-empty, connected, proper open subset of $\R^{1,n}$. Let $\Omega$ be a domain in $\R^{1,n}$, and let $x\in \Omega$ and $v\in T_x\Omega$ a lightlike vector. Let $x_v^+,x_v^-\in \R^{1,n}$ denote the future and past endpoints of the maximal lightlike segment tangent to $v$ and contained in $\Omega$, respectively. These points may be infinite if the maximal lightlike geodesic in $\Omega$ tangent to $v$ is complete in the future or in the past. 
\begin{prop}[{\cite[Prop. 4.5]{article_dist_markotitz}}]
    \label{formule_F_Omega}
    Let $\Omega$ be a domain in $\R^{1,n}$ and let $v\in T_x\Omega$ be lightlike. Then 
    $$F_\Omega(v)=\frac{\vert v \vert}{d_h(x,x_v^-)}+\frac{\vert v \vert}{d_h(x,x_v^+)},$$
    where $d_h(x,x_v^\pm)$ is understood to be infinite when $x_v^\pm$ is infinite. \qed
\end{prop}

\begin{prop}
    \label{prop_complétude_delta_Omega}
    Let $\Omega$ be a convex domain in $\R^{1,n}$. The following properties are equivalent:
    \begin{enumerate}
        \item $\delta_\Omega$ is a complete distance on $\Omega$;
        \item $\delta_\Omega$ is a distance on $\Omega$;
        \item all lightlike lines in $\Omega$ are incomplete.
    \end{enumerate}
\end{prop}

\begin{proof}
    This is a consequence of \cite[Thm. C]{article_dist_markotitz}.
\end{proof}

\section{Causal convexity and future completeness}\label{section_causal_convexité}

A domain $\Omega$ in $\R^{1,n}$ is \emph{causally convex} if any causal curve in $\R^{1,n}$ with endpoints in $\Omega$ is contained in $\Omega$. We say that $\Omega$ is \emph{future complete} if any future causal curve in $\R^{1,n}$ starting in $\Omega$ is entirely contained in $\Omega$. Any future complete domain is causally convex. 

\begin{fact}
    \label{fait_cc_implique_globalement_hyperbolique}
    Causally convex domains in $\R^{1,n}$ are globally hyperbolic. 
\end{fact}

\begin{proof}
    Let $\Omega$ be a causally convex domain in $\R^{1,n}$. Then any diamond in $\Omega$ is a diamond in $\R^{1,n}$. Therefore any diamond in $\Omega$ is compact. It follows that $\Omega$ is globally hyperbolic, see \cite{Geroch} and \cite{Sanchez2008}.
\end{proof}

\begin{prop}
    \label{prop_structure_ouverts_causallement_convexes}
    Let $\Omega$ be a domain in $\R^{1,n}$. Then $\Omega$ is causally convex if and only if there exist an open subset $U\subset \R^n$ and two 1-Lipschitz functions $f_\pm:\overline U\to \overline\R$ such that $f_-<f_+$ on $U$, $f_-=f_+$ on $\partial U$, and  
    $$\Omega=\{(t,p)\in\R\times U\vert\,f_-(p)<t<f_+(p)\}.$$
    Moreover, the domain $\Omega$ is future complete if and only if $f^+=+\infty$, the domain $\Omega$ is bounded if and only if $U$ is bounded in $\R^n$, and the domain $\Omega$ is convex if and only if $U$ and $f_-$ are convex and $f_+$ is concave. 
\end{prop}

\begin{proof}
    The proof of the ``if and only if'' part is identical to \cite[Sec. 3.5]{smaï2023enveloping}. Let us recall how the maps $f_+$ and $f_-$ are constructed, when $\Omega$ is causally convex. Let $\pi:\R^{1,n}\to \R^n$ be the orthogonal projection map, and let $U=\pi(\Omega)$. For every $p\in U$, the set $\Omega\cap\pi^{-1}(\{p\})$ takes the form $(f_-(p),f_+(p))\times \{p\}$ since $\Omega$ is causally convex. Let $p,q\in U$ and let $x=(f_+(p),p)$ and $y=(f_+(q),q)$. This defines two maps $f_\pm:U\to \overline\R$. By construction, the points $x$ and $y$ cannot be causally related, which implies that 
    $$0\leq \b(y-x,y-x)=-(f_+(p)-f_+(q))^2+\vert p-q\vert^2.$$
    Hence $\vert f_+(p)-f_+(q)\vert\leq\vert p-q\vert$, so $f_+$ is $1$-Lipschitz. Similarly, $f_-$ is $1$-Lipschitz. 
    The characterizations of future completeness, boundedness or convexity of $\Omega$ follow directly from the construction of $U$, $f_+$ and $f_-$.
\end{proof}

\begin{definition}
    Let $\Omega$ be a causally convex domain in $\R^{1,n}$, and let $f_+$ and $f_-$ be the two functions given by Proposition \ref{prop_structure_ouverts_causallement_convexes}. The \emph{future (resp. past) causal boundary of $\Omega$}, denoted $\partial_c^+\Omega$ (resp. $\partial_c^-\Omega$), is defined as the graph of $f_+$ (resp. $f_-$). By convention, we let $\partial_c^+\Omega=\emptyset$ if $f_+$ is infinite, and similarly for $\partial_c^-\Omega$.
\end{definition}

One can check that the future (resp. past) causal boundary points consist exactly of the endpoints of future (resp. past) inextensible causal curves in $\Omega$. Hence, the notion of future (resp. past) causal boundary is independent of the choice of splitting $\R^{1,n}=\R\times \R^n$. In fact, the future and past causal boundary are closely related to the causal boundary defined by Geroch--Kronheimer--Penrose. 

\begin{definition}[\cite{Penrose_Kronheimer_Geroch}]
    Let $(M,[g])$ be a globally hyperbolic spacetime. The \emph{future causal boundary of $M$} is the set of equivalence classes of future inextensible causal curve $\gamma:[0,+\infty)\to M$, under the relation $\gamma_1\sim \gamma_2$ if and only if $I^-(\gamma_1)=I^-(\gamma_2)$. The \emph{past causal boundary} is defined analogously. These boundaries are naturally endowed with the Alexandrov topology.
\end{definition}

\begin{fact}[\cite{Rym_causal_completion}]
    Let $\Omega$ be a causally convex domain in $\R^{1,n}$. Assume that there exists a stable past lightcone $C_\varepsilon^-$, such that $\Omega$ is contained in $I^+(C_\varepsilon^-)$. Then $\partial_c^-\Omega$ is naturally homeomorphic to the past causal boundary of $\Omega$ in the sense of Geroch--Kronheimer--Penrose. The same property holds for $\partial_c^+\Omega$ if signs are reversed.
\end{fact}

\begin{proof}
    Let $\gamma:[0,1)\to \Omega$ be an inextensible past causal curve in $\Omega$. Since $C_\varepsilon^-$ is a Cauchy hypersurface of $\R^{1,n}$, the curve $\gamma$ must accumulate at some point $y\in \R^{1,n}$. The map $[\gamma]\to \lim_{t\to 1}\gamma(t)$ gives the required homeomorphism, see \cite[Cor. 4.1]{Rym_causal_completion}.
\end{proof}

\subsection{Stable acausality}\label{section_acausalité_stable}
A subset $\Ccal$ of $\R^{1,n}$ is \emph{acausal/achronal} if two distinct points of $\Ccal$ are not causally/chronally related in $\R^{1,n}$.

\begin{definition}
    \label{def_stable_acausalité}
    Let $\varepsilon>0$ and let $\Ccal$ be a subset of $\R^{1,n}$. Then $\Ccal$ is \emph{$\varepsilon$-stably acausal} if $\Ccal$ is an acausal subset of $(\R^{1,n},[\b_\varepsilon])$. If there exists some $\varepsilon>0$ such that $\Ccal$ is $\varepsilon$-stably acausal, we say that $\Ccal$ is \emph{stably acausal}. 
\end{definition}

Note that $\varepsilon$-stable acausality is a notion that depends on the choice of splitting $\R^{1,n}=\R\times \R^n$, since the value of $\b_\varepsilon$ depends on that splitting. However, stable acausality is independent of the choice of splitting.

\subsection{Cosmological time functions} Let $\Omega$ be a causally convex domain in $\R^{1,n}$. The \emph{past and future cosmological time functions} on $\Omega$ are the maps $\tau^-:\Omega\to (0,+\infty]$ and $\tau^+:\Omega\to (0,+\infty]$ defined as 
$$\tau^\pm(x)=\sup\{\tauL(x,y)\,\vert\,y\in\Omega \text{ and } y\in J^\pm(x)\}.$$
The past cosmological time function is the standard cosmological time function, as studied in \cite{Andersson_Galloway_Howard_1998}. The future cosmological time function is not a time function in the standard sense, since it is decreasing along future causal curves. When $\Omega$ is future complete, then $\tau^+=+\infty$. In that case, we may write $\tau=\tau^-$ for the (past) cosmological time function of $\Omega$. Recall that for a conformal spacetime $(M,[g])$, the cosmological time function $\tau$ is said to be \emph{regular} if it is finite and if $\tau\to 0$ along any inextensible past causal curve. Regular times are in particular continuous, see \cite{Andersson_Galloway_Howard_1998}.

\begin{prop}
    \label{proposition_existence_initial_singularity}
    Let $\Omega$ be a causally convex domain in $\R^{1,n}$, and let $\varepsilon>0$. Assume that there exists a stable past lightcone $C^-_\varepsilon$, such that $\Omega$ is contained in $I^+(C_\varepsilon^-)$. Then $\tau^-$ is regular. Also, for every $x\in \Omega$, there exists at least one $y\in \partial_c^-\Omega$ such that $\tau^-(x)=\tauL(y,x)$. If $\Omega$ is convex, then $y$ is unique and the hyperplane orthogonal to $x-y$ at $y$ is a supporting hyperplane of $\Omega$ at $y$. The same holds for $\tau^+$ if signs are reversed.
\end{prop}

\begin{proof}
    Let $x\in \Omega$. Then $K=J^-(x)\cap J^+(C^-_\varepsilon)$ is compact since $C^-_\varepsilon$ is a Cauchy hypersurface of $\R^{1,n}$. Since $J^-(x)\cap \Omega\subset K$, we deduce that $\tau^-$ is finite, and that there exist $y\in \partial\Omega$ such that $\tau^-(x)=\tauL(y,x)$ and $y\leq x$. It follows that $y\in \partial_c^-\Omega$. Let $\gamma:[0,1)\to \Omega$ be a past inextensible causal curve starting at $x$, and let $z\in \partial_c^-\Omega$ be its endpoint. Let $t_k\to 1$ as $k\to \infty$ and $z_k\in \partial_c^-\Omega$ such that $\tau^-(\gamma(t_k))=\tauL(z_k,\gamma(t_k))$ for all $k\geq 0$. Since $z_k\in K$ for all $k\geq 0$, we can always extract a subsequence of $z_k$ and assume that $z_k\to z$, for some $z\in \partial_c^-\Omega$. Then $z\in J^-(y)\cap \overline{\Omega}\subset C^-(y)$, hence $\tauL(z,y)=0$. It follows that $\tau(\gamma)\to 0$ and that $\tau$ is regular. The convex case follows from \cite[Sec. 4]{Bonsante}.
\end{proof}

Following \cite{Andersson_Galloway_Howard_1998}, we say that a point $y\in \partial_c^\pm\Omega$ is a \emph{terminal/initial singularity} for $x\in \Omega$ if $\tau^\pm(x)=\tauL(y,x)$.

\begin{prop}
    \label{prop_initial_singularity}
    Let $\Omega$ be a causally convex domain in $\R^{1,n}$ such that $\partial_c^-\Omega$ is $\varepsilon$-stably acausal. Let $y\in \Omega$ and let $x\in \partial_c^- \Omega$ be an initial singularity for $y$. Then 
    $$y-x\in \{(t,p)\in \R^{1,n}\,\vert\,(1+\varepsilon)\vert p\vert \leq t\}.$$
\end{prop}

\begin{proof}
    Write $y-x=(t,p)$. Let $\Hcal=\{z\in I^-(y)\,\vert\,\tauL(z,y)=\tauL(x,y)\}$. Since $\Hcal$ is contained in  $I^-(\Omega)$, it is disjoint from $I^+_\varepsilon(x)$. It follows that
    $T_x\Hcal$
    is a spacelike or lightlike subspace of $(\R^{1,n},[\b_\varepsilon])$. Since $(T_x\Hcal)^{\perp_{\b_\varepsilon}}$ is spanned by $v=\left(\frac{t}{(1+\varepsilon)^2},p\right)$, we obtain 
    $$ \b_\varepsilon(v,v)=-\frac{t^2}{(1+\varepsilon)^2}+\vert p\vert^2\leq 0.$$
    Therefore $(1+\varepsilon)\vert p\vert \leq t$.
\end{proof}

\section{Comparison with the quasi-hyperbolic metric}\label{section_comparaison_avec_dist_quasi_hyp}

Recall that $h=dt^2+d\sigma^2$ is the Wick rotation of $\b$ along $\partial_t$. Let $\Omega$ be a domain in $\R^{1,n}$. The \emph{quasi-hyperbolic metric} of $\Omega$ is defined as
$$\hk_\Omega=\frac{1}{d_h(x,\partial\Omega)^2}h.$$ 
We will denote by $\K_\Omega$ its associated Riemannian distance, called \emph{the quasi-hyperbolic distance} of $\Omega$. This section is devoted to the proof of the following theorem.

\begin{thm}
    \label{thm_equiv_explicite_delta_Omega_et_qh}               
    Let $\Omega$ be a causally convex domain in $\R^{1,n}$ and let $\varepsilon>0$. Assume that $\partial_c^+ \Omega$ and $\partial_c^- \Omega$ are $\varepsilon$-stably acausal. Then  
    \begin{equation}
        \label{ineg_delta_Omega_et_qh}
        \frac{\varepsilon}{\sqrt{(2+\varepsilon)^2+\varepsilon^2}}\K_\Omega\leq \delta_\Omega \leq 2\sqrt{2}\K_\Omega.
    \end{equation}
\end{thm}

The proof is given at the end of this section. We start with the following lemma. 

\begin{lem}
    \label{lemme_general_metric_conf_plat_et_d_light}
    Let $\Omega$ be a domain in $\R^{1,n}$ and let $f:\Omega\to \R_{>0}$ be a continuous function. Let $\hk=f^2\cdot h$. For every $x\in \Omega$ and for every $\varepsilon>0$, there exists a neighborhood $V_x\subset \Omega$ of $x$ such that for every segment $[a,b]\subset V_x$, there exists a piecewise lightlike geodesic curve $\gamma$, arbitrarily close to $[a,b]$, such that 
    $$L_\hk(\gamma)\leq (\sqrt{2}+\varepsilon)L_\hk([a,b]).$$
\end{lem}

\begin{proof}
    Let $x\in\Omega$ and $\varepsilon>0$, and let $\varepsilon_0>0$ such that $\frac{f(x)+\varepsilon_0}{f(x)-\varepsilon_0}<1+\frac{\varepsilon}{\sqrt 2}$. Since $f$ is continuous, we can find a neighborhood $V_x$ of $x$ such that $\vert f(y)-f(x)\vert <\varepsilon_0$ holds for every $y\in V_x$. Let $[a,b]$ be an affine segment contained in $V_x$ and let $\Pi$ be a plane containing $[a,b]$ and $\partial_t$. Then $\Pi$ is a Lorentzian plane and is foliated by two lightlike line distributions $\Delta_+$ and $\Delta_-$ that are orthogonal with respect to $h$. Then, one can find a piecewise affine curve from $a$ to $b$ that is tangent to $\Delta_+$ or $\Delta_-$, arbitrarily close to $[a,b]$ and such that 
    $$L_h(\gamma)\leq \sqrt2L_h([a,b]).$$
    Then, by choosing $\gamma$ sufficiently close to $[a,b]$ so that $\gamma\subset V_x$, and since $(f(x)-\varepsilon_0)^2\hk\leq h\leq (f(x)+\varepsilon_0)^2\hk$ holds on $V_x$, one has 
    $$
    \begin{aligned}
        L_\hk(\gamma)&\leq (f(x)+\varepsilon_0)L_{h}(\gamma)\\
        &\leq \sqrt{2}(f(x)+\varepsilon_0)L_{h}([a,b])\\
        &\leq \sqrt{2}\frac{f(x)+\varepsilon_0}{f(x)-\varepsilon_0}L_\hk([a,b])\\
        &= (\sqrt{2}+\varepsilon)L_\hk([a,b]).    
    \end{aligned}
    $$
\end{proof}

\begin{definition}
    Let $\hk$ be a $C^0$ Riemannian metric on an open subset $\Omega\subset \R^{1,n}$. For every $x,y\in \R^{1,n}$, we let 
$$d_\hk^\light(x,y)=\inf_{x \rightsquigarrow y} L_\hk(\gamma),$$
where the infimum runs over all piecewise lightlike geodesic curves from $x$ to $y$. 
\end{definition}

\begin{cor}
    \label{cor_general_metric_conf_plat_et_d_light}
    Let $\Omega$ be a domain in $\R^{1,n}$ and let $\hk=f^2\cdot h$, where $f:\Omega\to \R_{>0}$ is continuous. Then 
    $$d_\hk\leq d_\hk^\light\leq\sqrt{2}\,d_\hk.$$
\end{cor}

\begin{proof}
    The first inequality is clear from the definition of $d_\hk$ and $d_\hk^\light$. Let $x,y\in \Omega$ and let $\varepsilon>0$. We can find a $C^1$ curve $\alpha:[0,1]\to \Omega$ such that $L_\hk(\alpha)<d_\hk(x,y)(1+\varepsilon)$. For $k\geq 1$, let $\alpha_k:[0,1]\to \Omega$ be the concatenation of the segments $[\alpha(\frac{i}{k}),\alpha(\frac{i+1}{k})]$. Then $\alpha_k\to \alpha$ in the $C^1$ topology as $k\to \infty$, hence we can find $k\geq 1$ such that 
    $$\vert L_\hk(\alpha)-L_\hk(\alpha_k)\vert<\varepsilon L_\hk(\alpha).$$
    Since the image of $\alpha_k$ is a compact subset of $\Omega$, we can find $r>0$ such that the restriction of $\alpha$ to any segment of length $\leq r$ takes values in a subset $V_x$ as in Lemma \ref{lemme_general_metric_conf_plat_et_d_light}. Therefore, there exists a piecewise lightlike curve $\gamma$, arbitrarily close to $\alpha_k$, such that $L_\hk(\gamma)\leq (\sqrt2+\varepsilon)L_\hk(\alpha_k)$. Then 
    $$
    \begin{aligned}
        d_\hk^\light(x,y)\leq L_\hk(\gamma)&\leq (\sqrt2+\varepsilon)(1+\varepsilon)L_\hk(\alpha)\\
        &\leq (\sqrt2+\varepsilon)(1+\varepsilon)^2d_\hk(x,y).
    \end{aligned}
    $$
    Since the previous inequality holds for any $\varepsilon>0$, we deduce that $d_\hk^\light(x,y)\leq \sqrt2d_\hk(x,y)$.
\end{proof}

\begin{prop}\phantomsection\label{prop_estimee}
    Let $\Omega$ be a future complete domain in $\R^{1,n}$. If $\partial\Omega$ is $\varepsilon$-stably acausal for some $\varepsilon>0$, then for every $x\in \Omega$ and $v\in \R^{1,n}$ lightlike, one has
    $$d_h(x,\partial\Omega)\leq d_h(x,x_v^-)\leq \frac{\sqrt{(2+\varepsilon)^2+\varepsilon^2}}{\varepsilon}d_h(x,\partial\Omega).$$
\end{prop}

The proof will use the following elementary fact:

\begin{fact}
    \label{tan_formula}
    Let $\Delta=\{y=ax\}$ and $\Delta^\prime=\{y=bx\}$ be two lines in $\R^2$ with $0<a<b$. Denote by $\theta$ the angle between $\Delta$ and $\Delta^\prime$. Then $\tan(\theta)=\frac{b-a}{1+ab}$.\qed
\end{fact}
    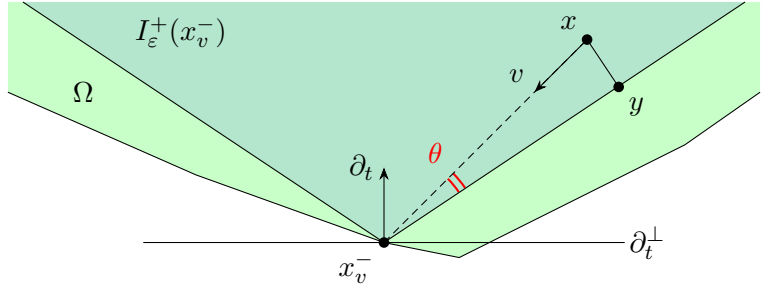
\begin{figure}
    \centering
    \begin{tikzpicture}[xscale=1]

    \def\a{3.2}
    \fill[bovert, opacity=1]  (-5,\a)
    -- (-5,2)
    -- (-2.5,0.9) 
    -- (0,0)  
    -- (1,-0.2) 
    -- (4,1.3)
    -- (5,2)
    -- (5,\a) ;
    \draw  (-5,2)
    -- (-2.5,0.9) 
    -- (0,0)  
    -- (1,-0.2) 
    -- (4,1.3)
    -- (5,2) ;
    \node at (-4,2) {$\Omega$};

    \def\B{1.5}
    \draw[fill=blue,opacity=0.1] (-\a*\B,\a) -- (0,0) -- (\a*\B,\a);

    \def\angle{45-11}
    \def\deltaa{11}
    \def\r{1.2}
     \draw[red, thick] ({\r*cos(\angle)}, {\r*sin(\angle)}) arc[start angle=\angle, end angle=\angle + \deltaa, radius=\r];
     \draw[red, thick] ({(\r+0.1)*cos(\angle)}, {(\r+0.1)*sin(\angle)}) arc[start angle=\angle, end angle=\angle + \deltaa, radius=\r+0.1]node [above left] {$\theta$};

    \draw (-\a*\B,\a) -- (0,0) -- (\a*\B,\a);

    \def\xun{2.7}
    \def\xdeux{2.7}
    \coordinate (x) at (\xun,\xdeux);
    \draw[densely dashed] (0,0) -- (x);
    \fill (x) circle (2pt) node [above left] {$x$};
    \draw[->,>=Stealth] (x) -- (\xun-0.7,\xdeux-0.7)node [above left] {$v$};

    \draw[->,>=Stealth] (0,0) -- (0,1);
    \draw (-0.3,1) node  {$\partial_t$};

    \draw (-\a,0) -- (\a,0);
    \node at (\a+0.3,0) {$\partial_t^\perp$};
    \node at (-2.7,2.8) {$I^+_\varepsilon(x_v^-)$};
    \fill (0,0) circle (2pt) node [below left] {$x_v^-$};

    \def\c{0.42}
    \coordinate (y) at (\xun +\c,\xdeux-\B*\c);
    \draw (y) -- (x);
    \fill (y) circle (2pt) node [below right] {$y$};

\end{tikzpicture}

    \caption{Argument in the proof of Proposition \ref{prop_estimee}.}
    \label{dessin_preuve_inegalité}
\end{figure}
\begin{proof}[Proof of Proposition \ref{prop_estimee}]
    Let $x\in \Omega$ and $v\in \R^{1,n}$ be lightlike. The first inequality is immediate since $x_v^-\in \partial\Omega$. For simplicity, assume that $x_v^-=0$. Since $\partial\Omega$ is $\varepsilon$-stably acausal, one has $I^+_\varepsilon(0)\subset \Omega$. Let 
    $\Delta_x=(\R \partial_t\times \R_{\geq 0}x)\cap C^+_{\varepsilon}(0),$
    and let $y\in \Delta_x$ be the $h$-orthogonal projection of $x$ onto $\Delta_x$. Then 
    $d_h(x,C_\varepsilon^+(0))=d_h(x,y).$
    Let $\theta$ denote the $h$-angle between $x$ and $y$. Then from Fact \ref{tan_formula}, one has
    $$\tan(\theta)=\frac{1-\frac{1}{1+\varepsilon}}{1+\frac{1}{1+\varepsilon}}=\frac{\varepsilon}{2+\varepsilon}.$$
    Hence $\sin(\theta)=\frac{\tan(\theta)}{\sqrt{1+\tan^2(\theta)}}=\frac{\varepsilon}{\sqrt{(2+\varepsilon)^2+\varepsilon^2}}$. The second inequality then follows, since 
    $$ d_h(x,\partial\Omega)\geq d_h(x,C^+_{\varepsilon}(0))=\sin(\theta)d_h(x,x_v^-).$$
\end{proof}

\begin{cor}\label{cor_estimée_delta_O_et_k_O}
    Let $\Omega$ be a causally convex domain in $\R^{1,n}$ that is neither future nor past complete. If $\partial_c^+ \Omega$ and $\partial_c^- \Omega$ are $\varepsilon$-stably acausal, then for every $x\in \Omega$ and $v\in \R^{1,n}$ lightlike, one has
    $$d_h(x,\partial\Omega)\leq \min\{d_h(x,x_v^-),d_h(x,x_v^+)\}\leq \frac{\sqrt{(2+\varepsilon)^2+\varepsilon^2}}{\varepsilon}d_h(x,\partial\Omega).$$
\end{cor}

\begin{proof}
    Let $x\in \Omega$ and $v\in \R^{1,n}$ be lightlike. Again, the first inequality is immediate. Since $\partial_c^+ \Omega$ and $\partial_c^- \Omega$ are $\varepsilon$-stably acausal, one has $I^+_\varepsilon(x_v^-,x_v^+)\subset \Omega$. 
    From the proof of Proposition \ref{prop_estimee}, one has 
    $$d_h(x,x_v^\pm)=\frac{\sqrt{(2+\varepsilon)^2+\varepsilon^2}}{\varepsilon}d_h(x,C_\varepsilon^\mp(x_v^\pm)).$$
    Therefore 
    $$
    \begin{aligned}
        \min\{d_h(x,x_v^-),d_h(x,x_v^+)\}&=\frac{\sqrt{(2+\varepsilon)^2+\varepsilon^2}}{\varepsilon}\min\{d_h(x,C_\varepsilon^+(x_v^-)),d_h(x,C_\varepsilon^-(x_v^+))\}\\
        &=\frac{\sqrt{(2+\varepsilon)^2+\varepsilon^2}}{\varepsilon}d_h(x,\partial I_\varepsilon(x_v^-,x_v^+))\\
        &\leq\frac{\sqrt{(2+\varepsilon)^2+\varepsilon^2}}{\varepsilon}d_h(x,\partial \Omega).\\
    \end{aligned}
    $$

\end{proof}

\begin{proof}[Proof of Theorem \ref{thm_equiv_explicite_delta_Omega_et_qh}]
    Let $x\in \Omega$ and $v\in T_x\Omega$ be lightlike. From Proposition \ref{prop_estimee} and Corollary \ref{cor_estimée_delta_O_et_k_O}, as well as from the formula given by Proposition \ref{formule_F_Omega}, one has 
    $$\frac{\varepsilon}{\sqrt{(2+\varepsilon)^2+\varepsilon^2}}\sqrt{\hk_\Omega(v,v)}\leq F_\Omega(v)=\frac{\vert v\vert}{d_h(x,x_v^-)}+\frac{\vert v\vert}{d_h(x,x_v^+)}\leq 2\frac{\vert v\vert}{d_h(x,\partial\Omega)}=2\sqrt{\hk_\Omega(v,v)}.$$
    From Theorem \ref{thm_lien_dmark_Fmark}, one deduces that 
    $$\frac{\varepsilon}{\sqrt{(2+\varepsilon)^2+\varepsilon^2}}d_{\hk_\Omega}^\light\leq \delta_\Omega\leq 2d_{\hk_\Omega}^\light.$$
    Since $d_h(\cdot,\partial\Omega)$ is continuous, Corollary \ref{cor_general_metric_conf_plat_et_d_light} implies that $\K_\Omega\leq d_{\hk_\Omega}^\light\leq \sqrt 2\K_\Omega$. The result follows.
\end{proof}

\begin{cor}
    \label{cor_completude_delta_Omega}               
    Let $\Omega$ be a causally convex domain in $\R^{1,n}$. If $\partial_c^+ \Omega$ and $\partial_c^- \Omega$ are stably acausal, then $(\Omega,\delta_\Omega)$ is a complete metric space.
\end{cor}

\begin{proof}
    This follows from Theorem \ref{thm_equiv_explicite_delta_Omega_et_qh} and the fact that $(\Omega,\K_\Omega)$ is a complete metric space, see \cite[Prop. 2.8]{Uniformizing_gromov_hyp_space}.
\end{proof}

\section{Stable acausality implies Gromov hyperbolicity}\label{section_stab_acaus_implique_Gromov_hyp}

The main goal of this section is to prove the following theorem.

\begin{thm}
    \label{thm_général_condi_suffisante}
    Let $\Omega$ be a future complete domain in $\R^{1,n}$ such that $\partial\Omega$ is stably acausal, or a bounded, convex and causally convex domain in $\R^{1,n}$ such that $\partial_c^+\Omega$ and $\partial_c^-\Omega$ are stably acausal. Then $(\Omega,\delta_\Omega)$ is Gromov hyperbolic.
\end{thm}

The ``if'' part of Theorems \ref{thm_intro_CNS_cc} and \ref{thm_intro_CNS_futur_complet} follows immediately from Theorem \ref{thm_général_condi_suffisante}. In the end of this section, we discuss the Gromov boundary, as well as a new notion of \emph{causally thin spacetimes}.

\subsection{Gromov hyperbolic metric spaces} We recall some terminology and properties of Gromov hyperbolic metric spaces, see for instance \cite[Chap. III.H]{bridson2013metric}. All of the following propositions are stated for proper and geodesic metric spaces, but hold for complete, locally compact length metric spaces, since those are always proper and geodesic by the Hopf-Rinow theorem, see \cite[Chap. I.3]{bridson2013metric}. 

\begin{definition}
    \label{def_gromov_hyp}
    A proper and geodesic metric space $(X,d)$ is \emph{Gromov hyperbolic} if there exists $\delta>0$ such that every geodesic triangle is \emph{$\delta$-thin}, that is, for every geodesic triangle $xyz\subset X$, the segment $[x,y]$ is contained in the $\delta$-neighborhood of $[x,z]\cup [y,z]$. In that case, $(X,d)$ is said to be \emph{$\delta$-Gromov hyperbolic}.
\end{definition} 

In a metric space $(X,d)$, the \emph{Gromov product} is defined as 
$$(x,y)_z = \frac{1}{2}\left(d(z,x) + d(z,y) - d(x,y)\right).$$

\begin{prop}[{\cite[Chap. III.H]{bridson2013metric}}]
    \label{prop_caracterisation_produit_gromov}
    Let $(X,d)$ be a proper and geodesic metric space. Then $(X,d)$ is Gromov hyperbolic if and only if there exists a constant $\delta>0$ such that $$(x,z)_w \ge \min\{(x,y)_w,(y,z)_w\} - \delta,$$
    for all $x,y,z,w\in X$.\qed
\end{prop}

Two distances $d_1$ and $d_2$ on a set $X$ are \emph{quasi-isometrically equivalent} if there exist constants $\alpha,\beta>0$ such that, for all $x,y\in X$, one has
    $$\frac{1}{\alpha} d_1(x,y)-\beta\leq d_2(x,y)\leq \alpha d_1(x,y)+\beta.$$

\begin{prop}[{\cite[Chap. III.H]{bridson2013metric}}]
    \label{prop_dist_quasi_iso}
    Let $d_1,d_2$ be two proper and geodesic distances on a set $X$. If $d_1$ and $d_2$ are quasi-isometrically equivalent,
    then $(X,d_1)$ is Gromov hyperbolic if and only if $(X,d_2)$ is Gromov hyperbolic.\qed
\end{prop}

 Let $(X,d)$ be a metric space. Given $A\geq 1$ and $B>0$, a \emph{$(A,B)$-quasi-geodesic} is a curve $\gamma:I\to (X,d)$, where $I$ is an interval of the real line, such that 
$$\frac{1}{A}\vert t-s\vert-B\leq d(\gamma(s),\gamma(t))\leq A\vert t-s\vert+B,$$
for all $s,t\in I$. 

\begin{prop}[{\cite[Chap. III.H]{bridson2013metric}}]
    \label{prop_stability_of_thinness}
    Let $(X,d)$ be a complete and geodesic $\delta$-Gromov hyperbolic metric space. Then, for every $A\geq 1$ and $B>0$, there exists a constant $\delta^\prime=\delta^\prime(\delta,A,B)>0$ such that every $(A,B)$-quasi-geodesic triangle in $X$ is $\delta^\prime$-thin.\qed
\end{prop}

One can deduce the following necessary condition for Gromov hyperbolicity.

\begin{prop}
    \label{lemme_triangle_non_fins}
    Let $(X,d)$ be a proper and geodesic Gromov hyperbolic metric space. Let $A\geq 1$ and $B\geq 0$ and let $x_ky_kz_k$ be a sequence of $(A,B)$-quasi-geodesic triangles. If $[x_k,y_k]$ and $[y_k,z_k]$ are eventually disjoint from every compact subset, then $[x_k,z_k]$ is eventually disjoint from every compact subset. 
\end{prop}

\begin{proof}
    From Proposition \ref{prop_stability_of_thinness}, let $\delta>0$ be such that every $(A,B)$-quasi-geodesic triangle in $X$ is $\delta$-thin. Let $K$ be a compact subset of $X$, and let $K^\prime=\mathcal{N}_{\delta+1}(K)$ be its $(\delta+1)$-neighborhood in $(X,d)$. Then $K^\prime$ is compact since $X$ is proper. Let $k\geq 0$ be large enough so that $[x_k,y_k]$ and $[y_k,z_k]$ are disjoint from $K^\prime$. Let $a\in [x_k,z_k]$ and let $b\in [x_k,y_k]\cup [y_k,z_k]$ such that $d(a,b)\leq \delta$. Then 
    $$d(a, K)\geq  d(b, K)-d(a,b)\geq d(b, K)-\delta>0,$$
    hence $a\not\in K$ and $[x_k,z_k]$ is disjoint from $K$.
\end{proof}

\subsection{Hyperbolicity of the quasi-hyperbolic metric} Let $\Omega\subset \R^{1,n}$ and let $x,y\in \Omega$. Let $\gamma$ be a piecewise smooth curve in $\Omega$ joining $x$ to $y$. Then $\gamma$ is called $\alpha$-\emph{uniform} for some constant $\alpha>0$ if the following two properties are satisfied:
\begin{itemize}
    \item $\gamma$ is parametrized by arc length with respect to $h=dt^2+d\sigma^2$ and $L_h(\gamma)\leq \alpha\, d_h(x,y)$,
    \item for all $t\leq L_h(\gamma)/2$, one has $B_h(\gamma(t),\frac{t}{\alpha})\subset \Omega$ and $B_h(\gamma(T-t),\frac{t}{\alpha})\subset \Omega$, where $B_h(a,r)$ denotes the $h$-ball of center $a$ and radius $r$.
\end{itemize}

The following definition can be made more general, see \cite{Uniformizing_gromov_hyp_space}. However, we state it only for domains in $\R^{1,n}$.

\begin{definition}
    A domain $\Omega$ in $\R^{1,n}$ is called \emph{uniform} if there exists $\alpha>0$, such that any two points in $\Omega$ can be joined by an $\alpha$-uniform curve. In that case, $\Omega$ is called \emph{$\alpha$-uniform}.
\end{definition}

  Note that for a fixed $\alpha>0$, the notion of $\alpha$-uniform domain depends on the choice of metric $h$, hence on the choice of the splitting of $\R^{1,n}=\R\times \R^n$. However, the notion of uniform domain is independent of such decomposition. 

\begin{thm}[{\cite[Thm. 1.11]{Uniformizing_gromov_hyp_space}}]
    \label{thm_hyperbolicité_dist_q_h}
    Let $\alpha>0$. If $\Omega$ is an $\alpha$-uniform domain in $\R^{1,n}$, then $(\Omega,\K_\Omega)$ is $\delta$-Gromov hyperbolic, for some constant $\delta=\delta(\alpha)$ depending only on $\alpha$.\qed
\end{thm}

\begin{prop}
    \label{prop_futur_complet_implique_uniforme}
    Any future complete domain in $\R^{1,n}$ is $\frac{3+\sqrt{10}}{2}$-uniform.
\end{prop}

\begin{proof}
    Let $\Omega$ be a future complete domain and let $x=(p,s)\in\Omega$ and $y=(q,t)\in \Omega$.
    
    \emph{Case 1:  $x$ and $y$ are causally related}. We can assume for instance that $x\leq y$. Let $z=(p,t)$ and $\gamma$ be the union of the segments $[x,z]$ and $[z,y]$. Then $\gamma$ has length $L(\gamma)\leq \sqrt{2}d(p,q)$ and $\gamma$ is a $\sqrt{2}$-uniform curve from $x$ to $y$. 
    
     \emph{Case 2: $x$ and $y$ are not causally related}. We can assume for instance that $t\geq s$, so that $d(p,q)\geq t-s$. Let $z=(q,t+\frac{3}{2}d(p,q))$ and let $\gamma$ be the union of the segments $[x,z]$ and $[z,y]$. Then 
        $$
            L(\gamma)=L([x,z])+L([z,y])
            = \sqrt{d(p,q)^2+\left(\frac{3}{2}d(p,q)\right)^2}+\frac{1}{2}d(p,q)+t-s
            \leq \frac{3+\sqrt{10}}{2}d(p,q).
        $$
        Also, since $f$ is 1-Lipschitz, one has $B=B(z,\frac{d(p,q)}{2\sqrt{2}})\subset \Omega$, and the convex envelopes $\Conv(x,B)$ and $\Conv(y,B)$ are contained in $\Omega$. If $m$ refers to the midpoint of $\gamma$, then 
        $$\frac{L(\gamma)/2}{L([x,z])}B\left(m,\frac{d(p,q)}{2\sqrt{2}}\right)\subset \Omega.$$
        Hence $B(m,r)\subset \Omega$, where 
        $$r=\frac{L(\gamma)d(p,q)}{L([x,z])4\sqrt{2}}=\frac{3+\sqrt{10}}{4\sqrt{20}}d(p,q)\geq \frac{2}{3+\sqrt{10}}d(p,q).$$
        Hence $\gamma$ is $\frac{3+\sqrt{10}}{2}$-uniform. It follows from both cases that $\Omega$ is $\frac{3+\sqrt{10}}{2}$-uniform.   
\end{proof}


\begin{thm}[{see for instance \cite[Thm. 2.19]{Vaisala}}]
\label{thm_convex_implique_uniforme}
Any bounded convex domain in $\R^{1,n}$ is uniform. More precisely, if $\Omega$ is convex and if $r_1,r_2>0$ are such that $B_h(x,r_1)\subset\Omega\subset B_h(x,r_2)$ for some $x\in \Omega$, then $\Omega$ is $\frac{r_2}{r_1}$-uniform. \qed
\end{thm}

\begin{proof}[Proof of Theorem \ref{thm_général_condi_suffisante}]
    Let $\Omega$ be a bounded, convex, causally convex domain in $\R^{1,n}$, or a future complete domain in $\R^{1,n}$. By Proposition \ref{prop_futur_complet_implique_uniforme} and Theorem \ref{thm_convex_implique_uniforme}, the domain $\Omega$ is uniform. By Theorem \ref{thm_hyperbolicité_dist_q_h}, the metric space $(\Omega,\K_\Omega)$ is Gromov hyperbolic. Assume that $\partial_c^+\Omega$ and $\partial_c^-\Omega$ are stably acausal. Then, by Theorem \ref{thm_equiv_explicite_delta_Omega_et_qh}, the metric space $(\Omega,\delta_\Omega)$ is bi-Lipschitz isometric to $(\Omega,\K_\Omega)$. Hence $(\Omega,\delta_\Omega)$ is Gromov hyperbolic by Proposition \ref{prop_dist_quasi_iso}.
\end{proof}

\subsection{Gromov's boundary vs the causal boundary} In this section, we compare the Gromov boundary and the causal boundary. 

\begin{definition}[Gromov boundary]
    Let $(X,d)$ be a proper and geodesic metric space. The \emph{Gromov boundary of $X$}, denoted $\partial_\infty X$, is the set of equivalence classes of geodesic rays $\gamma:[0,+\infty)\to X$, under the relation $\gamma_1\sim \gamma_2 $ if and only if there exists $M$ such that $d(\gamma_1(t),\gamma_2(t))<M$ for all $t\geq 0$. 
\end{definition}

The Gromov boundary $\partial_\infty X$ is naturally given a topology of uniform convergence on compact subsets. With this topology, the union $X\cup \partial_\infty X$ is a compact topological space.

\begin{prop}[{\cite[Chap. III.H]{bridson2013metric}}]
    \label{prop_identification_bord_Gromov}
    If $d_1$ and $d_2$ are two proper, geodesic and quasi-isometrically equivalent distances on a set $X$, then the identity map $(X,d_1)\to (X,d_2)$ extends to a homeomorphism of their Gromov boundary. 
\end{prop}

\begin{prop}
    Let $\Omega$ be a bounded, convex, causally convex domain in $\R^{1,n}$ (resp. future complete domain in $\R^{1,n}$) such that $\partial_c^+\Omega$ and $\partial_c^-\Omega$ are stably acausal. Then the identity map $\Omega \to \Omega$ extends to an embedding
    $$\Omega\cup \partial_c\Omega\to\Omega\cup\partial_\infty\Omega,$$
    where $\partial_\infty\Omega$ is the Gromov boundary of $(\Omega,\delta_\Omega)$ and $\partial_c\Omega=\partial_c^+\Omega\cup \partial_c^-\Omega$.
\end{prop}

\begin{proof}
    From Proposition \ref{prop_identification_bord_Gromov} and Theorem \ref{thm_equiv_explicite_delta_Omega_et_qh}, we can always replace $\delta_\Omega$ by the quasi-hyperbolic distance. From \cite[Prop. 3.12]{Uniformizing_gromov_hyp_space}, the identity map extends to a homeomorphism $\Omega\cup \partial_\star\Omega\to\Omega\cup\partial_\infty\Omega$, where $\partial_\star\Omega=\partial\Omega$ if $\Omega$ is bounded, and $\partial_\star\Omega$ is the one point compactification of $\partial\Omega$ if $\Omega$ is future complete. The result follows.
\end{proof}

Note that Bonk--Heinonen--Koskela \cite{Uniformizing_gromov_hyp_space} show that for an unbounded uniform domain $\Omega$ in $\R^{1,n}$, any geodesic ray $\gamma$ of $(\Omega,\K_\Omega)$ escaping every compact subset of $\Omega$ must converge to the same point of $\partial_\infty \Omega$. This has the following consequence.

\begin{prop}
    Let $\varepsilon>0$ and let $\Omega$ be a future complete domain in $\R^{1,n}$, such that $\partial\Omega$ is $\varepsilon$-stably acausal. Then there exists a constant $M=M(\varepsilon)>0$ such that, for every $x\in \Omega$ and for every pair of future causal curves $\gamma_1$ and $\gamma_2$ starting at $x$, then 
    $\delta_\Omega(\gamma_1,\gamma_2)\leq M.$
    The same property is satisfied by $\dhat_{\ln(\tau)}$, where $\tau$ is the cosmological time function of $\Omega$.
\end{prop}

\begin{proof}
    From Theorem \ref{thm_hyperbolicité_dist_q_h} and Proposition \ref{prop_futur_complet_implique_uniforme}, there exists $r>0$, independent of $\Omega$, such that  $(\Omega,\K_\Omega)$ is $r$-Gromov hyperbolic. 
    From Theorem \ref{thm_equivalence_d_lnt_delta_Omega_cas_stablement_acausal} and Theorem \ref{thm_equiv_explicite_delta_Omega_et_qh}, there exists $A_\varepsilon,B_\varepsilon>0$ such that any causal curve in $\Omega$ is a $(A_\varepsilon,B_\varepsilon)$-quasi-geodesic of $(\Omega,\K_\Omega)$. Let $x\in \Omega$ and let $\gamma_1,\gamma_2:[0,+\infty)\to \Omega$ be two inextensible future causal curves starting at $x$. Then $\gamma_1$ and $\gamma_2$ have the same endpoint in $\partial_\infty\Omega$ since they have infinite $h$-length, see \cite[Prop. 3.12]{Uniformizing_gromov_hyp_space}. Hence, by Proposition \ref{prop_stability_of_thinness}, there exists a constant $M=M(\varepsilon)$ such that $\K_\Omega(\gamma_1,\gamma_2)\leq M(\varepsilon)$. The same property holds for $\delta_\Omega$ and $\dhat_{\ln(\tau)}$, since all these distances are quantitatively quasi-isometrically equivalent by Theorem \ref{thm_equivalence_d_lnt_delta_Omega_cas_stablement_acausal} and Theorem \ref{thm_equiv_explicite_delta_Omega_et_qh}.
\end{proof}

\subsection{Causally thin spacetimes} \label{Section_causally_thin} In full generality, Theorem \ref{thm_général_condi_suffisante} does not hold for arbitrary causally convex domains, even if they are assumed convex, see Example \ref{ex_spacelike_slab}. Nonetheless, stable acausality still implies a certain form of thinness.

\begin{definition}[Causally thin spacetimes]
    Let $(M,[g])$ be a conformal spacetime endowed with a distance $d$. We say that $(M,[g])$ is \emph{causally thin} if causal diamonds are uniformly thin in the following sense: there exists $r>0$ such that every causal diamond $J(x,y)$ is contained in the $r$-neighborhood of any causal curve joining $x$ to $y$.
\end{definition}

\begin{rmk}
    The previous definition has the following interpretation: in a causally thin spacetime ($(M,[g],d)$, two causal curves that meet at some event and meet again later cannot separate arbitrarily far from each other in between. Their mutual distance, measured with respect to $d$, is uniformly bounded, independently of their trajectories and of their initial and terminal meeting points.
\end{rmk} 

\begin{thm}
    \label{thm_causally_thin}
    Let $\Omega$ be a causally convex domain in $\R^{1,n}$. If $\partial_c^-\Omega$ and $\partial_c^+\Omega$ are stably acausal, then $(\Omega,\delta_\Omega)$ is causally thin.
\end{thm}

The proof is based on Lemma \ref{lemme_2_0_quasi_geod_grace_a_la_null_dist} and on the following lemma.

\begin{lem}
    \label{lemme_1_I_eps}
    For every $\varepsilon>0$, there exists a constant $r=r(\varepsilon)$ such that, for every $a,b\in \R^{1,n}$ satisfying $a\leq b$, the metric space $({I_\varepsilon(a,b)},\delta_{I_\varepsilon(a,b)})$ is $r$-Gromov hyperbolic.
\end{lem}

\begin{proof}
    Let $\varepsilon>0$ and let $K=\{v\in I^+(0)\,\vert\, h(v,v)=1\}$. For $v\in K$, the stable diamond $I_\varepsilon(0,v)$ is a non-empty bounded convex domain in $\R^{1,n}$. Hence, there exists $s\geq 1$ such that, there exist $x\in I_\varepsilon(0,v)$ and $\rho>0$ satisfying 
    $$B_h(x,\rho)\subset I_\varepsilon(0,v)\subset B_h(x,s\rho).$$
    Let $s_v$ be the least $s\geq 1$ having this property. The map $f:v\in K\mapsto s_v$ is positive and continuous. Since $K$ is compact, the map $f$ admits a minimum $s_{\min}>0$. Let $a,b\in \R^{1,n}$ such that $a\leq b$. There exists a similarity $g$ with trivial isometry part such that $g(a)=0$ and $g(b)\in K$. Let $x\in I_\varepsilon(0,g(b))$ and $\rho>0$ such that
    $B_h(x,\rho)\subset I_\varepsilon(0,g(b))\subset B_h(x,s_{\min}\rho).$
    Then, if $y=g^{-1}(x)$, one has $$B_h(y,\rho/\lambda)\subset I_\varepsilon(a,b)\subset B_h(y,s_{\min}\rho/\lambda),$$
    where $\lambda$ is the distortion of $g$. From Theorem \ref{thm_convex_implique_uniforme}, the domain $I_\varepsilon(a,b)$ is $s_{\min}$-uniform. The result follows from Theorem \ref{thm_hyperbolicité_dist_q_h}.
\end{proof}

\begin{proof}[Proof of Theorem \ref{thm_causally_thin}]
    Let $\varepsilon>0$ such that $\partial_c^+\Omega$ and $\partial_c^-\Omega$ are $\varepsilon$-stably acausal.
    By Proposition \ref{prop_stability_of_thinness} and Lemma \ref{lemme_1_I_eps}, there exists $r=r(\varepsilon)$ such that for every $a,b\in \R^{1,n}$, any $(2,0)$-quasi-geodesic triangle in  $({I_\varepsilon(a,b)},\delta_{I_\varepsilon(a,b)})$ is $r$-thin. 
     Let $\beta,\gamma$ be two causal curves in $\Omega$ with the same endpoints. Let $a,b\in \Omega$ such that $I(a,b)$ contains the image of $\gamma$. The stable diamond $I_\varepsilon(a,b)$ is contained in $\Omega$ since $\partial_c^+\Omega$ and $\partial_c^-\Omega$ are $\varepsilon$-stably acausal. By Lemma \ref{lemme_2_0_quasi_geod_grace_a_la_null_dist} below, the triangle $\beta\cup\gamma$ is a $(2,0)$-quasi-geodesic bigon of $({I_\varepsilon(a,b)},\delta_{I_\varepsilon(a,b)})$. Hence 
    $\delta_{I_\varepsilon(a,b)}(\beta,\gamma)\leq r.$
      By Proposition \ref{proposition_naturalité}, we deduce
    $$\delta_\Omega(\beta,\gamma)\leq \delta_{I_\varepsilon(a,b)}(\beta,\gamma)\leq r.$$
    It follows that $(\Omega,\delta_\Omega)$ is causally thin.
\end{proof}

\section{Comparison with the null distance}\label{section_comparaison_avec_null_dist} 

\subsection{Null distance on spacetimes} A \emph{time function} on conformal spacetime $(M,[g])$ is a continuous function $\tau:M\to \R$ that is strictly increasing along future causal curves. 
\begin{definition}[{\cite{Sormani_vega}}]
    \label{def_null_distance}
    Let $(M,[g])$ be a conformal spacetime endowed with a time function~$\tau$. The \emph{null distance} between two points $x,y\in M$ is defined as 
    $$\dhat_\tau(x,y)=
\inf_{\substack{x_1,\dots,x_m}}
\sum_{k=1}^m\vert\tau(x_{i+1})-\tau(x_i)\vert,$$
where the infimum runs over all finite sequences of points such that $x_1=x,\dots,x_m=y$ and $x_k$ and $x_{k+1}$ are causally related for all $k$.
\end{definition}
The function $\dhat_\tau$ always defines a pseudodistance on $M$, but it may fail to be definite, see for instance \cite[Prop. 3.4]{Sormani_vega}.

\begin{fact}
    \label{fait_null_distance}
Let $(M,[g])$ be a conformal spacetime endowed with a time function $\tau$. Let $x,y\in M$ such that $x\leq y$. Then
    $\dhat_\tau(x,y)=\tau(y)-\tau(x).$\qed
\end{fact}

\begin{fact}[{\cite[Rmk. 3.7]{Sormani_vega}}]
\label{fact_def_null_distance}
    Definition \ref{def_null_distance} is unchanged if the infimum is taken over all finite sequences $x_1,\dots,x_m$ such that two consecutive points are joined by a lightlike geodesic.\qed
\end{fact}

\subsection{General comparisons}\label{section_comparaison_generale_Mark_vs_null} Our goal in this section is to compare the Markowitz pseudodistance to the null distance in the context of conformal spacetimes.
Although these two distances have similar definitions -- based on different notions of length of piecewise lightlike geodesic curves -- let us mention the following differences between them:
\begin{itemize}
    \item (Definition) $\delta_M$ is defined for any spacetime, and in fact on any pseudo-Riemannian manifold, while $\dhat_\tau$ requires the existence and the choice of a time function $\tau$.
    \item (Regularity) $\delta_M$ requires that the conformal class $[g]$ contains a $C^2$ regular metric, since projective parameters are defined in terms of the Ricci curvature.  In contrast, $\dhat_\tau$  can be defined for low regularity spacetimes, see \cite{Kunzinger_Steinbauer}. 
    \item (Invariance) $\delta_M$ is invariant under the full conformal group $\Conf(M)$, while $\dhat_\tau$ is in general only invariant under the subgroup $\Conf_\tau(M)\subset \Conf(M)$ of conformal transformations preserving the $\tau$-length of causal curves. 
    \item (Computability) Computing $\delta_M$ is in general quite challenging. In contrast, the value of the null distance between  two causally related points is explicit and all causal curves are geodesic for $\dhat_\tau$. 
\end{itemize}

The following two lemmas serve as general comparisons between these two distances.

\begin{lem}
    \label{lemme_comparaison_delta_M_et_d_t}
    Let $(M,[g])$ be a conformal spacetime endowed with a time function $\tau$. Assume that there exists $\alpha>0$, such that for every $x,y\in M$ that lie on a common lightlike geodesic, one has
    \begin{equation}
        \dhat_\tau(x,y)\leq \alpha\,\delta_M(x,y).
    \end{equation}
    Then $\dhat_\tau\leq \alpha\,\delta_M$. The same conclusion holds if the role of $\dhat_\tau$ and $\delta_M$ are reversed. 
\end{lem}

\begin{proof}
     Let $x,y\in M$ and let $\Ccal=\{(\gamma_k),(s_k),(t_k)\}$ is a lightlike chain from $x$ to $y$. Write $x_k=\gamma_k(s_k)$ for all $k$. Since $x_k$ and $x_{k+1}$ lie on a common lightlike geodesic, one has 
    $$\dhat_\tau(x,y)\leq \sum_{k=1}^m\dhat_\tau(x_{k-1},x_k)
        \leq \sum_{k=1}^m \alpha \,\delta_M(x_{k-1},x_k)
        \leq \alpha\sum_{k=1}^m  \rho_{I}(s_k,t_k)=\alpha L(\Ccal).$$ 
    Taking the infimum over all chains $\Ccal$ from $x$ to $y$, we deduce that $\dhat_\tau(x,y)\leq \alpha\,\delta_M(x,y)$. 
    
    Conversely, assume that $\delta_M(x,y)\leq \alpha\,\dhat_\tau(x,y)$ holds for every $x,y\in M$ that lie on a common lightlike geodesic. Let $x,y\in M$ and let $x_0=x,\dots,x_m=y$ be a finite sequence of points such that $x_k$ and $x_{k+1}$ are on a common lightlike geodesic. Then 
    $$\delta_M(x,y)\leq \sum_{k=1}^m \delta_M(x_{k-1},x_k)\leq \sum_{k=1}^m \alpha \,\dhat_\tau(x_{k-1},x_k)=\alpha\sum_{k=1}^m  \,\dhat_\tau(x_{k-1},x_k). $$
    Taking the infimum over all finite sequences $x_0,\dots,x_m$ such that consecutive points are on a common lightlike geodesic, we deduce that $\delta_M(x,y)\leq \alpha\,\dhat_\tau(x,y)$ by Fact \ref{fact_def_null_distance}. 
\end{proof}

A time function $\tau:M\to \R$ on a conformal spacetime $(M,[g])$ is \emph{Cauchy} if its level sets are empty or Cauchy hypersurfaces of $(M,[g])$.

\begin{lem}
    \label{lemme_temps_projectif}
    Let $(M,[g])$ be a globally hyperbolic conformal spacetime endowed with a projective Cauchy time function $\tau:M\to \R$. If $J=\tau(M)$ is a strict interval of $\R$, then  
    \begin{equation}
        \delta_M=\dhat_{f\circ\tau}
    \end{equation}
    where $f:(J,\rho_J)\to \R$ is any isometry onto the Euclidean line. 
\end{lem}

\begin{proof}
    From Lemma \ref{lemme_comparaison_delta_M_et_d_t}, it is sufficient to check the previous equality when $x$ and $y$ are on a common lightlike geodesic. 
    Let $\gamma$ be a maximal lightlike geodesic and let $\gamma:J\to M$ be its $\tau$-parametrization. Let $x,y\in \gamma$ and write $s=\tau(x)$ and $t=\tau(y)$.  From Proposition \ref{Proposition_applications_projectives}, one has 
    $\rho_J(\tau(x),\tau(y))\leq \delta_M(x,y).$ Since $\gamma$ is projectively parameterized, one has
    $$\delta_M(x,y)\leq \rho_J(s,t)=\rho_J(\tau(x),\tau(y)).$$
    Hence $\delta_M(x,y)=\rho_J(\tau(x),\tau(y))=\dhat_{f\circ\tau}(x,y).$
\end{proof}

\subsection{Future complete case} We return to our setting in Minkowski space.

\begin{thm}
    \label{thm_equivalence_d_lnt_delta_Omega}
    Let $\Omega$ be a convex future complete domain in $\R^{1,n}$ containing no lightlike line. Let $\tau$ be its cosmological time function. Then
            \begin{equation}
                \dhat_{\ln(\tau)}\leq \delta_\Omega\leq 2\dhat_{\ln(\tau)}.
                \label{ineg_delt_d_ln}
            \end{equation}
            In addition, equality holds on the left-hand side of Equation (\ref{ineg_delt_d_ln}) (resp. on the right-hand side) if and only if $\Omega$ is the future of a spacelike hyperplane (resp. the future of a point).
\end{thm}

\begin{proof}
    Let us first show that equality holds on the right-hand side of Equation (\ref{ineg_delt_d_ln}) when $\Omega$ is the future of a point. Let $a\in \R^{1,n}$ and let $\Omega=I^+(a)$. In that case, the point $a$ is the only initial singularity, and the cosmological time function is given by $\tau(x)=\tauL(a,x)$ for all $x\in \Omega$. Let $f:\Omega \to \R_{>0}$ be defined by $f(x)= \tauL(a,x)^2$. Let $\gamma(t)=b+tv$ be an affine lightlike geodesic in $\Omega$. Then 
    $$f(\gamma(t))=\b(b-a,b-a)+2\b(b-a,v)t.$$
    Thus, the map $f$ is projective. Therefore $f$ is a projective Cauchy time function.
    From Lemma \ref{lemme_temps_projectif}, one has 
    $$\delta_\Omega=\dhat_{\ln(f)}=2\dhat_{\ln(\tau)}.$$

    We now show that equality holds on the left-hand side of Equation (\ref{ineg_delt_d_ln}) when $\Omega$ is the future of a spacelike hyperplane. Let $H$ be a spacelike hyperplane and let $\Omega=I^+(H)$. Let $p_H$ denote the orthogonal projection onto $H$. Then the cosmological time function of $\Omega$ is given by $\tau(x)=\tauL(p_H(x),x)$. This function $\tau$ is a Cauchy time function that is projective, since it is in fact linear. Hence, from Lemma \ref{lemme_temps_projectif}, it follows that
    $\delta_\Omega=\dhat_{\ln(\tau)}.$

    We now turn to the general case. Since $\Omega$ contains no complete lightlike line, it admits at least one spacelike supporting hyperplane. By Proposition \ref{proposition_existence_initial_singularity}, the cosmological time function $\tau$ is finite, and every $x\in \Omega$ admits a unique initial singularity $r(x)$. Let $x,y\in \Omega$ such that $x\leq y$. Since $r(x)\ll x$, the future of $r(x)$ contains $x$ and $y$. Since $I^+(r(x))\subset \Omega$, Proposition \ref{proposition_naturalité} and the previous computation imply that
    $$\delta_\Omega(x,y)\leq\delta_{I^+(r(x))}(x,y)=2\ln(\tauL(r(x),y))-2\ln(\tauL(r(x),x)).$$
    Now, since $r(x)\leq y$, we have $\tauL(r(x),y)\leq \tau(y)$ by definition of $\tau$. Hence
    $$\delta_\Omega(x,y)\leq 2\ln(\tau(y))-2\ln(\tau(x))= 2\dhat_{\ln(\tau)}(x,y).$$
    Let $H$ denote the hyperplane containing $r(x)$ and orthogonal to $x-r(x)$ with respect to $\b$, and let $p_H$ denote the orthogonal projection onto $H$. By Proposition \ref{proposition_existence_initial_singularity}, the domain $\Omega$ is contained in $I^+(H)$, hence
    $$\delta_\Omega(x,y)\geq\delta_{I^+(H)}(x,y)=\ln(\tauL(p_H(y),y))-\ln(\tauL(p_H(x),x)).$$
    Since $r(y)\in J^+(H)\cap J^-(y)$, we have $\tau(y)\leq \tauL(p_H(y),y)$. Since $r(x)=p_H(x)$, we obtain 
    $$\delta_\Omega(x,y)\geq \ln(\tau(y))-\ln(\tau(x))=\dhat_{\ln(\tau)}(x,y).$$
    This shows that Equation (\ref{ineg_delt_d_ln}) holds for any $x\leq y$. From Lemma \ref{lemme_comparaison_delta_M_et_d_t}, we deduce that Equation (\ref{ineg_delt_d_ln}) holds globally on $\Omega$. 

    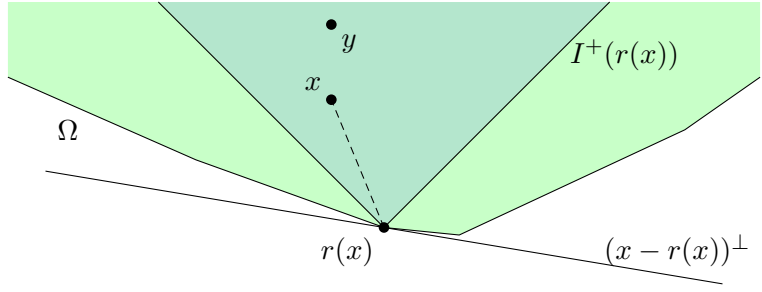
\begin{figure}
        \centering
            \begin{tikzpicture}[xscale=1]

    \def\a{3}
    \fill[bovert, opacity=1]  (-5,\a)
    -- (-5,2)
    -- (-2.5,0.9) 
    -- (0,0)  
    -- (1,-0.1)
    -- (4,1.3)
    -- (5,2)
    -- (5,\a) ;
    \draw  (-5,2)
    -- (-2.5,0.9) 
    -- (0,0)  
    -- (1,-0.1)
    -- (4,1.3)
    -- (5,2) ;
    \node at (-4.2,1.3) {$\Omega$};

    \def\B{1}
    \draw[fill=blue,opacity=0.1] (-\a*\B,\a) -- (0,0) -- (\a*\B,\a);

    \draw (-\a*\B,\a) -- (0,0) -- (\a*\B,\a);

    \def\xun{-0.7}
    \def\xdeux{1.7}
    \coordinate (x) at (\xun,\xdeux);
    \draw[densely dashed] (0,0) -- (x);
    \fill (x) circle (2pt) node [above left] {$x$};

    \draw (-1.5*\a,1.5*0.5) -- (1.5*\a,0-1.5*0.5);
    \node at (\a+0.9,-0.3) {$(x-r(x))^\perp$};
    \node at (3.2,2.3) {$I^+(r(x))$};
    \fill (0,0) circle (2pt) node [below left] {$r(x)$};

    \def\c{0.42}
    \coordinate (y) at (\xun ,\xdeux +1);

    \fill (y) circle (2pt) node [below right] {$y$};

\end{tikzpicture}
        \caption{Argument in the proof of Theorem \ref{thm_equivalence_d_lnt_delta_Omega}.}
        \label{fig_preuve_thm_equiv_nulldist_Mark}
    \end{figure}

    Let us show that if equality holds on the right-hand side of Equation (\ref{ineg_delt_d_ln}) for all $x,y\in\Omega$, then $\Omega$ is the future of a point. Indeed, under that hypothesis, for all $x,y\in \Omega$, the inequality $\tauL(r(x),y)\leq \tau(y)$ must be an equality, so $r(x)=r(y)$. In particular, the map $r$ is locally constant, hence it is constant and equal to some $p\in\partial \Omega$. It then follows that $\Omega=I^+(p)$. A similar argument shows that equality holds on the left-hand side of Equation (\ref{ineg_delt_d_ln}) only if $\Omega$ is the future of a spacelike hyperplane.
\end{proof}


\begin{thm}
    \label{thm_equivalence_d_lnt_delta_Omega_cas_stablement_acausal}
    Let $\Omega$ be a future complete domain in $\R^{1,n}$. If $\partial \Omega$ is $\varepsilon$-stably acausal, then 
    $$\frac{\varepsilon}{\sqrt{2+2(1+\varepsilon)^2}}\dhat_{\ln(\tau)}-c_\varepsilon\leq \delta_\Omega\leq 2\dhat_{\ln(\tau)}. $$
    where $c_\varepsilon=\frac{\varepsilon}{\sqrt{2+2(1+\varepsilon)^2}}\ln\left( \frac{(1+\varepsilon)^2+1}{(1+\varepsilon)^2-1}\right)$.
\end{thm}

The proof of Theorem \ref{thm_equivalence_d_lnt_delta_Omega_cas_stablement_acausal} is based on the following three lemmas. Let $\varepsilon>0$ and let 
$$\Omega_{\varepsilon}=\R^{1,n}\setminus J^-_{\varepsilon}(0)=\{x=(t,p)\in\R^{1,n}\,\vert\,\vert p\vert<(1+\varepsilon)t\}.$$
The domain $\Omega_{\varepsilon}$ is future complete and has stably acausal boundary. For $x\in \Omega_{\varepsilon}$ such that $x\not\in \R \partial_t$, we let 
$\Delta_x=(\R \partial_t\times \R_{\geq 0}x)\cap C^-_{\varepsilon}(0).$

\begin{lem}
    \label{temps_cosmo_Omega_eps}
    The cosmological time function of $\Omega_\varepsilon$ is given by 
    $\tau(t,p)=\frac{(1+\varepsilon)t+\vert p\vert}{\sqrt{(1+\varepsilon)^2-1}}.$
\end{lem}

\begin{proof}
    Let $x=(t,p)\in \R^{1,n}$. Assume first that $p\neq 0$. Let $y\in C_\varepsilon^-(0)$ be an initial singularity of $x$. Then $y\in \Hcal=\{z\in I^-(x)\,\vert\,\tauL(z,x)=\tau(x)\}$ and $\Hcal\subset J^-_\varepsilon(0)$, so $\Hcal$ and $C_\varepsilon^-(0)$ cannot intersect  transversely at $y$. Hence 
    $T_y\Hcal=T_yC_\varepsilon^-(0)$, which implies that $y\in \Delta_x$ and that $y$ is the Lorentzian orthogonal projection of $x$ onto $\Delta_x$. In particular, one can write $y=\lambda(-\vert p\vert,(1+\varepsilon)p)$, for some $\lambda>0$. Since $y$ and $x-y$ are orthogonal, one obtains
    $$0=\b(y,x-y)=\lambda\vert p\vert(\lambda \vert p\vert +t)-\lambda(1+\varepsilon)((1+\varepsilon)\lambda-1)\vert p\vert^2.$$
    Solving the previous equation gives $\lambda=\frac{t+(1+\varepsilon)\vert p\vert}{((1+\varepsilon)^2-1)\vert p\vert}$. We deduce that
    $$\tau(x)=\tauL(y,x)=\frac{(1+\varepsilon)t+\vert p\vert}{\sqrt{(1+\varepsilon)^2-1}}.$$
    The previous formula extends to arbitrary values of $p$ by continuity of $\tau$.
\end{proof}

\begin{lem}
    \label{lemme_Omega_eps}
    For all $x\in \Omega_{\varepsilon}$ and $v\in T_x\Omega_\varepsilon$ lightlike, one has 
    $d_h(x,x_v^-)\leq \frac{\sqrt{2(1+\varepsilon)^2-2}}{\varepsilon}\tau(x).$
\end{lem}

\begin{proof}
    Let $x=(t,p)\in \Omega_\varepsilon$. Assume first that $x\not \in C^+(0)$, so that $C(x)$ intersects $C^-_\varepsilon(0)$ transversely in a smooth  $(n-1)$-dimensional spacelike sphere denoted $\Sigma$. Let $f:\Sigma\to \R$ be the function given by $f(y)=d_h(x,y)$. Let $y\in \Sigma$ be a critical point of $f$. Then 
    $$\nabla_h f(y)\in T_y\Sigma^{\perp_h}=T_yC_\varepsilon^-(0)^{\perp_h}\cap T_yC(x)^{\perp_h},$$
    where $\nabla_h f$ denotes the $h$-gradient of $f$. Write $y=(s,q)$. The preceding equation implies 
    $$\left(\begin{array}{c}
        s-t \\
        q-p 
    \end{array}\right)\in \Span\left( \left(\begin{array}{c}
        s \\
        -(1+\varepsilon)^2q 
    \end{array}\right) , \left(\begin{array}{c}
        s-t \\
        -(q-p) 
    \end{array}\right) \right).$$
    It follows that $p$ and $q$ are collinear, and that $f$ admits a unique maximum and minimum, respectively. The maximum of $f$ is attained at the point $y\in \Sigma$ such that $p$ and $q-p$ are positively proportional. In other words, the maximum of $f$ is attained at the intersection between $\Delta_x$ and $\Delta=\{q=-s+t-p\}$. By denoting $y=(s,q)$ for the point realizing the maximum of $f$, then 
    $$\begin{cases}
        \vert q\vert=-s+t+\vert p\vert,\\
        \vert q\vert = -(1+\varepsilon)s.
    \end{cases}$$
    It follows that $s=-(t+\vert p\vert)/\varepsilon$ and 
    $$d_h(x,y)^2=2(s-t)^2=2\frac{\left((1+\varepsilon)t+\vert p\vert\right)^2}{\varepsilon^2}.$$
    From Lemma \ref{temps_cosmo_Omega_eps}, we deduce that 
    $$\sup_\Sigma f=\frac{\sqrt 2((1+\varepsilon)t+\vert p\vert)}{\varepsilon}=\frac{\sqrt{2(1+\varepsilon)^2-2}}{\varepsilon}\tau(x).$$
    The desired inequality follows since $x_v^-\in \Sigma$ for all lightlike vectors $v\in T_x\Omega_\varepsilon$.
\end{proof}

\begin{lem}
    \label{cor_Omega_eps}
    Let $\tau$ be the cosmological time function of $\Omega_\varepsilon$. Then 
    $\dhat_{\ln(\tau)}\leq \frac{\sqrt{2+2(1+\varepsilon)^2}}{\varepsilon}\delta_{\Omega_\varepsilon}.$
\end{lem}

\begin{proof}
    Let $x=(t,p)\in \Omega_\varepsilon$ with $p\neq 0$ and let $v$ be a lightlike vector. From Proposition \ref{formule_F_Omega} and Lemma \ref{lemme_Omega_eps}, one has
    $$T_x\ln(\tau)\cdot v=\frac{h(\nabla_h\tau(x),v)}{\tau(x)}\leq \frac{\vert\nabla_h\tau(x)\vert\,\vert v\vert}{\tau(x)}\leq \vert\nabla_h\tau(x)\vert\frac{\sqrt{2(1+\varepsilon)^2-2}}{\varepsilon}F_{\Omega_\varepsilon}(v).$$
    Now, by Lemma \ref{temps_cosmo_Omega_eps}, one has  
    $$\vert\nabla_h\tau(x)\vert=\frac{\vert(1+\varepsilon,p/\vert p \vert)\vert}{\sqrt{(1+\varepsilon)^2-1}}=\frac{\sqrt{1+(1+\varepsilon)^2}}{\sqrt{(1+\varepsilon)^2-1}},$$
    hence
    $$T_x\ln(\tau)\cdot v\leq \frac{\sqrt{2+2(1+\varepsilon)^2}}{\varepsilon}F_{\Omega_\varepsilon}(v).$$
    Let $x,y\in \Omega_\varepsilon$ such that $x\leq y$, and let $\gamma:[0,1]\to \Omega_\varepsilon$ be a piecewise lightlike geodesic curve from $x$ to $y$. Then 
        $$
        \dhat_{\ln(\tau)}(x,y)=\ln(\tau(y)/\tau(x))=\int_0^1\ln(\tau(\gamma(t))^\prime dt\leq \frac{\sqrt{2+2(1+\varepsilon)^2}}{\varepsilon}\int_0^1F_{\Omega_\varepsilon}(\gamma^\prime(t))dt.
    $$
    Taking the infimum over all $\gamma$, the result follows from Theorem \ref{thm_lien_dmark_Fmark} and Lemma \ref{lemme_comparaison_delta_M_et_d_t}.
\end{proof}

\begin{proof}[Proof of Theorem \ref{thm_equivalence_d_lnt_delta_Omega_cas_stablement_acausal}]
    Assume now that $\partial\Omega$ is $\varepsilon$-stably acausal. Note that in the proof of Theorem \ref{thm_equivalence_d_lnt_delta_Omega}, only future completeness was needed to prove $\delta_\Omega\leq 2\dhat_{\ln(\tau)}$. Therefore the right-hand side inequality holds since $\Omega$ is future complete. Let $x,y\in \Omega$ such that $x\leq y$, and let $r(x)\in\partial\Omega$ be an initial singularity for $x$. We denote $U=\R^{1,n}\setminus J^-_\varepsilon(r(x))$ and $\tau_U$ for the cosmological time function of $U$. Then $\Omega\subset U$ since $\partial\Omega$ is $\varepsilon$-stably acausal. From Proposition \ref{proposition_naturalité} and Corollary \ref{cor_Omega_eps}, one has 
    $$\frac{\varepsilon}{\sqrt{2+2(1+\varepsilon)^2}}\ln(\tau_U(y)/\tau_U(x))\leq \delta_U(x,y)\leq \delta_\Omega(x,y).$$
    Now, since $\Omega\subset U$, one has $\tau_U(y)\geq \tau(y)$. Write $x-r(x)=(t,p)$. From Proposition \ref{prop_initial_singularity}, one has $(1+\varepsilon)\vert p\vert\leq t$. It follows that
        $$
        (1+\varepsilon)t+\vert p\vert\leq \frac{(1+\varepsilon)^2+1}{1+\varepsilon}t,$$
        and $$\frac{(1+\varepsilon)^2-1}{(1+\varepsilon)^2}t^2=\left(1-\frac{1}{(1+\varepsilon)^2}\right)t^2\leq\left(t^2-\vert p\vert^2\right)=\tau(x)^2,$$
        hence 
        $$(1+\varepsilon)t+\vert p\vert\leq \frac{(1+\varepsilon)^2+1}{1+\varepsilon}\sqrt{\frac{(1+\varepsilon)^2}{(1+\varepsilon)^2-1}}\tau(x) 
        = \frac{(1+\varepsilon)^2+1}{\sqrt{(1+\varepsilon)^2-1}}\tau(x).$$
    By Lemma \ref{temps_cosmo_Omega_eps}, we obtain
    $$\tau_U(x)=\frac{(1+\varepsilon)t+\vert p\vert}{\sqrt{(1+\varepsilon)^2-1}}
                 \leq \frac{(1+\varepsilon)^2+1}{(1+\varepsilon)^2-1}\tau(x).$$
    Finally 
    $$ \frac{\varepsilon}{\sqrt{2+2(1+\varepsilon)^2}}\left(\ln\frac{\tau(y)}{\tau(x)}-\ln\left( \frac{(1+\varepsilon)^2+1}{(1+\varepsilon)^2-1}\right)\right)\leq \delta_\Omega(x,y).$$
    The inequality follows by Lemma \ref{lemme_comparaison_delta_M_et_d_t}.
\end{proof}

\subsection{Causally convex case} We perform similar estimates to those above for causally convex domains.

\begin{thm}
    \label{thm_equivalence_mark_nulldistance_cas_cc}
    Let $\Omega$ be a convex, causally convex domain in $\R^{1,n}$. Assume that $\Omega$ is contained between two spacelike hyperplanes. Let $\tau^-$ and $\tau^+$ be the associated past and future cosmological time functions. Then 
    $$\frac{1}{2}\dhat_{\ln(\tau^-/\tau^+)}\leq \delta_\Omega \leq 2\dhat_{\ln(\tau^-/\tau^+)}.$$
\end{thm}

\begin{proof}
    Let us prove that equality holds on the right-hand side when $\Omega=I(a,b)$ for some $a,b\in \R^{1,n}$ with $a\ll b$. In that case, one has $\tau^+(x)=\tauL(x,b)$ and $\tau^-(x)=\tauL(a,x)$ for all $x\in \Omega$. Let 
    $\tau:\Omega\to \R$ be the time function given by 
    $$\tau(x)=\frac{\tauL(a,x)^2}{\tauL(x,b)^2}.$$
    Then, if $\gamma(t)=x+tv$ is an affine lightlike geodesic, then 
    $$\tau(\gamma(t))=\frac{\tauL(a,x)^2+2\b(x-a,v)t}{\tauL(x,b)^2+2\b(x-b,v)t}$$
    is a homography, so $\tau$ is projective. Also, if $\gamma$ is an inextensible causal curve of $I(a,b)$, then $\gamma$ has a past and future endpoint in $C^+(a)$ and $C^-(b)$, respectively. Hence $\tau$ intersects every level set $\tau^{-1}(\{t\})$ for $t>0$, so $\tau$ is a Cauchy time function. Now $\tau(\Omega)=\R_{>0}$, hence from Lemma \ref{lemme_temps_projectif}, one has 
    \begin{equation}
    \label{equation_d_T_diamant}
        \delta_\Omega=\dhat_{\ln(\tau)}=2\dhat_{\ln(\tau^-/\tau^+)}.
    \end{equation}
    
    We now return to the general case. For $x\in \Omega$, we let $r^+(x)$ and $r^-(x)$ denote the terminal and initial singularity of $x$, which are given by Proposition \ref{proposition_existence_initial_singularity}. We let $H^\pm(x)$ denote the supporting hyperplane at $r^\pm(x)$ orthogonal to $r^\pm(x)-x$, and we let $p_x^\pm:\R^{1,n}\to H^\pm(x)$ denote the orthogonal projection onto $H_\pm(x)$. Let $x,y\in \Omega$ such that $x\leq y$. As in the proof of Theorem \ref{thm_equivalence_d_lnt_delta_Omega}, one has 
    \begin{align*}
        \delta_\Omega(x,y)&\geq \delta_{I^+(H^-(x))}(x,y)\\
        &=\ln(\tauL(p_x^-(y),y))-\ln(\tauL(p_x^-(x),x))\\
        &\geq \ln(\tau^-(y))-\ln(\tau^-(x)).
    \end{align*}
    Similarly, one has 
    $$\delta_\Omega(x,y)\geq \delta_{I^-(H^+(y))}(x,y)\geq \ln(\tau^+(x))-\ln(\tau^+(y)).$$
    Combining the previous two inequalities, we deduce
    $$\dhat_{\ln(\tau^-/\tau^+)}(x,y)=\ln(\tau^-(y))-\ln(\tau^+(y))-\ln(\tau^-(x))+\ln(\tau^+(x))\leq  2\delta_\Omega(x,y).$$
    Let $a=r^-(x)$ and $b=r^+(y)$ and let $D=I(a,b)$. Then $x,y\in D\subset \Omega$, hence 
    $$\begin{aligned}
        \delta_\Omega(x,y)&\leq \delta_D(x,y)\\
        &=\ln\left(\frac{\tauL(a,y)^2}{\tauL(y,b)^2}\frac{\tauL(x,b)^2}{\tauL(a,x)^2}\right)\\
        &=2\ln\left(\frac{\tauL(a,y)}{\tau^+(y)}\frac{\tauL(x,b)}{\tau^-(x)}\right)\\
        &\leq 2\ln\left(\frac{\tau^-(y)}{\tau^+(y)}\frac{\tau^+(x)}{\tau^-(x)}\right)=2\dhat_{\ln(\tau^-/\tau^+)}(x,y).
    \end{aligned}$$
    
    Hence, the desired inequalities hold for every $x\leq y$. Therefore, the inequalities hold for general $x,y\in \Omega$ by Lemma \ref{lemme_comparaison_delta_M_et_d_t}. 
\end{proof}

\begin{thm}
    \label{thm_equivalence_mark_nulldistance_cas_cc_bord_stablement_acausal}
    Let $\Omega$ be a causally convex domain in $\R^{1,n}$, that is neither future nor past complete. Let $\tau^-$ and $\tau^+$ be the associated past and future cosmological time functions. If $\partial_c^+ \Omega$ and $\partial_c^- \Omega$ are $\varepsilon$-stably acausal, then 
    $$\frac{\varepsilon}{2\sqrt{2+2(1+\varepsilon)^2}}\dhat_{\ln(\tau^-/\tau^+)}-c_\varepsilon\leq \delta_\Omega \leq 2\dhat_{\ln(\tau^-/\tau^+)}.$$
\end{thm}

\begin{proof}
    The inequality $\delta_\Omega \leq 2\dhat_{\ln(\tau^-/\tau^+)}$ still holds because $\Omega$ is causally convex. From Proposition \ref{proposition_existence_initial_singularity}, the cosmological functions are finite since regular. Let $x,y\in \Omega$ such that $x\leq y$. We choose an initial singularity $r(x)\in I^-(x)$ and we let $U=\R^{1,n}\setminus J_\varepsilon^-(r(x))$ and $\tau_U$ be the cosmological time function of $U$. Then by Lemma \ref{cor_Omega_eps}, one has
    $$\frac{\varepsilon}{\sqrt{2+2(1+\varepsilon)^2}}\ln(\tau_U(y)/\tau_U(x))\leq \delta_U(x,y) \leq\delta_\Omega(x,y) .$$
    A similar computation as in the proof of Theorem \ref{thm_equivalence_d_lnt_delta_Omega_cas_stablement_acausal} then implies that 
        $$\frac{\varepsilon}{\sqrt{2+2(1+\varepsilon)^2}}\ln\frac{\tau^-(y)}{\tau^-(x)}-c_\varepsilon\leq \delta_\Omega(x,y).$$
    A similar argument applied to a terminal singularity in the future of $y$ implies
    $$\frac{\varepsilon}{\sqrt{2+2(1+\varepsilon)^2}}\ln\frac{\tau^+(x)}{\tau^+(y)}-c_\varepsilon\leq \delta_\Omega(x,y).$$
    Combining the previous two inequalities gives
    $$\frac{\varepsilon}{2\sqrt{2+2(1+\varepsilon)^2}}\dhat_{\ln(\tau^-/\tau^+)}(x,y)-c_\varepsilon\leq \delta_\Omega(x,y).$$
    The result follows by Lemma \ref{lemme_comparaison_delta_M_et_d_t}.
\end{proof}

\subsection{Consequences for the null distance}\label{section_consequence_null_distance} In this section we prove Theorem \ref{thm_intro_null_dist}.

\begin{prop}\label{prop_completude_null_dist}
    Let $\Omega$ be a future complete (resp. causally convex and neither future nor past complete) domain in $\R^{1,n}$ and let $\tau$ be its cosmological time function (resp. $\tau=\tau^-/\tau^+$). If  $\partial_c^+\Omega$ and $\partial_c^-\Omega$ are stably acausal, or if $\Omega$ is convex and contains no lightlike line, then $(\Omega,\dhat_{\ln(\tau)})$ is a complete metric space. 
\end{prop}

\begin{proof}
    From Proposition \ref{prop_complétude_delta_Omega} and Corollary \ref{cor_completude_delta_Omega}, the metric space $(\Omega,\delta_\Omega)$ is complete. By Theorems \ref{thm_equivalence_d_lnt_delta_Omega}, \ref{thm_equivalence_d_lnt_delta_Omega_cas_stablement_acausal}, \ref{thm_equivalence_mark_nulldistance_cas_cc} and \ref{thm_equivalence_mark_nulldistance_cas_cc_bord_stablement_acausal}, we deduce that $(\Omega,\dhat_{\ln(\tau)})$ is a complete metric space.
\end{proof}

\begin{rmk}
    Burtsher--Garc{\'\i}a-Heveling have shown in \cite{Burtscher_Garcia_Heveling} that a conformal spacetime $(M,[g])$ is globally hyperbolic if and only if it admits a time function $\tau$ such that $(M,\dhat_\tau)$ is a complete metric space. If $M=(\Omega,[\b])$ is as in the previous corollary, then $M$ is globally hyperbolic, see Fact \ref{fait_cc_implique_globalement_hyperbolique}. The time function $\ln(\tau)$, where $\tau$ is as above, is a concrete example that illustrates Burtsher--Garc{\'\i}a-Heveling's theorem \cite{Burtscher_Garcia_Heveling}.
\end{rmk} 

\begin{proof}[Proof of Theorem \ref{thm_intro_null_dist}]
    Let $\Omega$ be a convex domain that is either bounded and causally convex, or future complete with no lightlike line, and let $\tau$ be the associated time function, as in Proposition \ref{prop_completude_null_dist}. From Proposition \ref{prop_completude_null_dist}, the metric space $(\Omega,\dhat_{\ln(\tau)})$ is complete. By Theorems \ref{thm_equivalence_d_lnt_delta_Omega}, \ref{thm_equivalence_d_lnt_delta_Omega_cas_stablement_acausal},  \ref{thm_equivalence_mark_nulldistance_cas_cc} and  \ref{thm_equivalence_mark_nulldistance_cas_cc_bord_stablement_acausal}, the distance $\delta_\Omega$ and $\dhat_{\ln(\tau)}$  are quasi-isometrically equivalent.
    By Proposition \ref{prop_dist_quasi_iso}, we deduce that $(\Omega,\delta_\Omega)$ is Gromov hyperbolic if and only if $(\Omega,\dhat_{\ln(\tau)})$ is Gromov hyperbolic. 
\end{proof}

\begin{rmk}
    Note that the previous proof shows that Theorem \ref{thm_général_condi_suffisante} and \ref{thm_causally_thin} also have analogues for the null distance. Precisely, these theorems remain valid if $\delta_\Omega$ is replaced by $\dhat_{\ln(\tau)}$, or by $\dhat_{\ln(\tau^-/\tau^+)}$, depending on whether $\Omega$ is future complete, or causally convex and neither future nor past complete.
\end{rmk}

\section{Gromov hyperbolicity implies stable acausality}\label{section_Gromov_hyp_implique_stab_acaus}

In this section, we prove the ``only if'' part of Theorems \ref{thm_intro_CNS_cc} and \ref{thm_intro_CNS_futur_complet}.

\subsection{Limits of convex domains} Recall that compact subsets of $\R^{1,n}$ are endowed with a natural topology induced by the Hausdorff distance, which can be expressed as
$$d_\text{Haus}(C,C^\prime)=\max_{x\in C^\prime}d_h(x,C)+\max_{y\in C}d_h(y,C^\prime).$$
We  will say that a sequence $(\Omega_k)$ of convex domains in $\R^{1,n}$ \emph{converges} to a convex domain $\Omega$ if for every compact $C\subset \R^{1,n}$, the sequence of compact subsets $C\cap \overline\Omega_k$ converges to  $C\cap \overline\Omega$ with respect to the Hausdorff distance.

\begin{lem}
\label{lemme_convergence_distance_markowitz_dun_convexe}
    Let $(\Omega_k)$ and $\Omega$ be convex domains in $\R^{1,n}$ containing no lightlike line. If $(\Omega_k)$ converges to $\Omega$, then 
    $\delta_{\Omega_k}\to \delta_{\Omega}$
    uniformly on compact subsets of $\Omega$. More precisely, for every compact $K\subset\Omega$, for every $\varepsilon>0$, there exists $k_0$ such that for every $k\geq k_0$, for every $x,y\in K$, one has 
    $$(1-\varepsilon)\delta_{\Omega_k}(x,y)\leq\delta_{\Omega}(x,y)\leq(1+\varepsilon)\delta_{\Omega_k}(x,y) $$
\end{lem}

\begin{proof}
    The proof is identical to \cite[App. B]{Limbeek_Zimmer}. 
\end{proof}

We deduce the following corollary, which is similar to \cite[Thm. 5.1]{Zimmer_complex_1}.

\begin{cor}
    \label{cor_limite_de_domaines_Gromov_hyp}
    Let $(\Omega_k)$ and $\Omega$ be convex domains in $\R^{1,n}$ containing no lightlike line, such that $(\Omega_k)$ converges to $\Omega$, and let $r>0$. If $(\Omega_k, \delta_{\Omega_k})$ is $r$-Gromov hyperbolic for all $k\geq 0$, then $(\Omega,\delta_\Omega)$ is Gromov hyperbolic.
\end{cor}

\begin{proof}
    Since $\Omega_k$ is $r$-Gromov hyperbolic, there exists a constant $r\geq 0$ such that 
    $$(x,z)_w^k \ge \min\{(x,y)_w^k,(y,z)_w^k\} - r,$$
    for all $x,y,z,w\in \Omega_k$, where $(\cdot\vert\cdot)^{(k)}$ is the Gromov product of $(\Omega_k, \delta_{\Omega_k})$. Taking the limit when $k\to \infty$, Lemma \ref{lemme_convergence_distance_markowitz_dun_convexe} implies that 
    $$(x,z)_w \ge \min\{(x,y)_w,(y,z)_w\} - r,$$
    for all $x,y,z,w\in \Omega$. It follows from Proposition \ref{prop_caracterisation_produit_gromov} that $(\Omega,\delta_\Omega)$ is Gromov hyperbolic.
\end{proof}

\subsection{A weaker necessary condition} A curve that is the union of two lightlike segments is called a \emph{broken lightlike segment}. Such a curve is either causal or achronal.

\begin{thm}
    \label{thm_coindi_necessaire_faible}
    Let $\Omega$ be a convex, causally convex domain in $\R^{1,n}$. If $(\Omega,\delta_\Omega)$ is Gromov hyperbolic, then $\partial\Omega$ does not contain a causal broken lightlike segment. If moreover $\Omega$ is future complete, then $\partial\Omega$ is acausal. 
\end{thm}

\begin{lem}
    \label{lemme_2_0_quasi_geod_grace_a_la_null_dist}
    Let $\Omega$ be a convex, future complete domain in $\R^{1,n}$, or a bounded, convex and causally convex domain in $\R^{1,n}$.  Let $\tau$ be the cosmological time function if $\Omega$ is future complete, and let $\tau=\tau^-/\tau^+$ otherwise. Then every causal curve $\gamma(t)$ parametrized by $t=\ln(\tau)$ is a $(2,0)$-quasi-geodesic of $(\Omega,\delta_\Omega)$. 
\end{lem}

\begin{proof}
    Let $\gamma(t)$ be a parametrized causal curve such that $t=\ln(\tau(\gamma))$. Then $\gamma$ is a metric space geodesic of $(\Omega,\dhat_{\ln(\tau)})$ by Fact \ref{fait_null_distance}. Hence,  Theorem \ref{thm_equivalence_d_lnt_delta_Omega} and Theorem \ref{thm_equivalence_mark_nulldistance_cas_cc} imply
    $$\frac{1}{2}\vert t-s\vert=\frac{1}{2}\dhat_{\ln(\tau)}(\gamma(s),\gamma(t)) \leq \delta_\Omega(\gamma(s),\gamma(t))\leq 2\dhat_{\ln(\tau)}(\gamma(s),\gamma(t))=2\vert t-s\vert, $$
    for all $s,t\in \R$.
\end{proof}

\begin{cor}
    \label{lemme_2_0_quasi_geodesiques_dans_un_ccc}
    Let $\Omega$ be a convex, causally convex domain in $\R^{1,n}$ containing no lightlike line. Then causal curves in $\Omega$ are $(2,0)$-quasi-geodesics of $(\Omega,\delta_\Omega)$.
\end{cor}

\begin{proof}
    Let $\gamma:[a,b]\to \Omega$ be a causal curve in $\Omega$. We claim that $\gamma$ is a rectifiable curve of $(\Omega,\delta_\Omega)$. Indeed $\gamma$ is a rectifiable curve of $(\R^{1,n},d_h)$, since it is a locally Lipschitz curve. Now $d_h$ and $\delta_\Omega$ are locally equivalent \cite[Thm. A]{article_dist_markotitz}, so $\gamma$ is a rectifiable curve of $(\Omega,\delta_\Omega)$. For all $k\geq 0$, let $x_k=(-2^k,0)$ and $y_k=-x_k$, and let 
    $\Omega_k=\Omega\cap I(x_k,y_k).$
    Since diamonds of $\R^{1,n}$ are convex, causally convex and bounded, it follows that $\Omega_k$ is convex, causally convex and bounded for all $k\geq 0$. By construction, the sequence $(\Omega_k)$ converges to $\Omega$ in the Hausdorff topology. 
    By Lemma \ref{lemme_convergence_distance_markowitz_dun_convexe}, one has
    $L_{\delta_{\Omega_k}}(\gamma)\to L_{\delta_{\Omega}}(\gamma),$
    as $k\to\infty$.
    By Lemma \ref{lemme_2_0_quasi_geod_grace_a_la_null_dist}, one has 
    $$\frac{1}{2}L_{\delta_{\Omega_k}}(\gamma)\leq \delta_{\Omega_k}(\gamma(a),\gamma(b))\leq 2L_{\delta_{\Omega_k}}(\gamma),$$
    for all sufficiently large $k\geq 0$.
    Taking the limit as $k\to \infty$ in the previous equation, we deduce that 
    $$\frac{1}{2}L_{\delta_{\Omega}}(\gamma)\leq \delta_{\Omega}(\gamma(a),\gamma(b))\leq 2L_{\delta_{\Omega}}(\gamma),$$
    hence $\gamma$ is a $(2,0)$-quasi-geodesic.
\end{proof}

\begin{proof}[Proof of Theorem \ref{thm_coindi_necessaire_faible}]
    Assume first that $\Omega$ is future complete. Assume by contradiction that $\partial\Omega$ is not acausal. Then there exists a half-line $\Delta$ in $\partial\Omega$. Let $p\in \Delta$ and let $x_k=p+2^{-k}\partial_t$, for all $k\geq 0$. Let $(y_k)$ be a sequence of points such that $y_k$ belongs to the future lightlike half-line parallel to $\Delta$ and containing $\Omega$, and such that $d_h(x_k,y_k)\to \infty$ as $k\to \infty$. Finally, for all $k\geq 0$, let $z_k$ be the unique point, distinct from $x_k$, that belongs to $C(y_k)\cap(p+\R\partial_t)$. Then the triangle $x_ky_kz_k$ is a causal triangle, hence it is a $(2,0)$-quasi-geodesic triangle of $(\Omega,\delta_\Omega)$ by Lemma \ref{lemme_2_0_quasi_geod_grace_a_la_null_dist}. By construction, the segments $[x_k,y_k]$ and $[y_k,z_k]$ leave every compact subset of $\Omega$ but $[x_k,z_k]$ does not. By Proposition \ref{lemme_triangle_non_fins}, $(\Omega,\delta_\Omega)$ is not Gromov hyperbolic.

    Assume now that $\Omega$ is causally convex and that $\partial\Omega$ contains a causal broken lightlike segment. Then one can construct similarly a sequence $x_ky_kz_k$ of causal triangles, such that $[x_k,y_k]$ and $[y_k,z_k]$ are eventually disjoint from every compact subset of $\Omega$, but $[x_k,z_k]$ is not, see Figure \ref{fig_non_gromov_hyp_cc}. These triangles are $(2,0)$-quasi-geodesic triangles of $(\Omega,\delta_\Omega)$ by Corollary \ref{lemme_2_0_quasi_geodesiques_dans_un_ccc}, hence $(\Omega,\delta_\Omega)$ is not Gromov hyperbolic.
\end{proof}

\begin{figure}
\centering
\captionsetup{width=.4\textwidth}
\begin{minipage}{.5\textwidth}
  \centering

    \begin{tikzpicture}[scale=0.8]

    \draw[fill=bovert, opacity=0.7]
    plot[smooth cycle, tension=0.3] coordinates {
      (0,0) (3,0) (6,3) (5,6) (2,3) (-1,3)
    };

        \coordinate (X) at (-1.5,4.5);
        
        \coordinate (Xp) at (-1.5,1);
        
        \begin{scope}[shift={(4.1,0.7)}]
            \draw[thick, cyan, opacity=1] (-1.5,1) -- (1.5,4);
        
        \coordinate (X) at (-1.5,4.5);
        
        \coordinate (Xp) at (-1.5,1);
        
        \draw[thick] (X) -- (Xp);

        \coordinate (a) at (-1.5,1.5);
        \coordinate (b) at (-1.5,4);
        \coordinate (c) at (-0.25,2.75);
        \draw[very thick,red] (a) -- (b) -- (c) -- cycle;
        \fill[red] (a) circle (2pt)node[above left] {$x_k$};
        \fill[red] (b) circle (2pt)node[above left] {$z_k$};
        \fill[red] (c) circle (2pt)node[above right] {$y_k$};

        \node at (2.1,4) {$\partial\Omega$};
        \end{scope}

\end{tikzpicture}
    \caption{\small{A non-thin family of causal geodesic triangles of $(\Omega,\delta_\Omega)$, inside a future complete domain $\Omega$, whose boundary contains a lightlike half-line (in blue).}}
    \label{fig_non_gromov_hyp_fc}

\end{minipage}%
\begin{minipage}{.5\textwidth}
  \centering

    \begin{tikzpicture}[scale=1.5]

    \draw[fill=bovert, opacity=0.7]
  (0,1) -- (1,0) -- (0,-1);
  \draw[very thick, cyan, opacity=1] (0,1) -- (1,0) -- (0,-1);
    
        \draw[fill=bovert, opacity=0.7] plot[smooth, tension=0.3] coordinates { (0,1)  (-1,1)  (-2,0.7) (-3.2,-0.2) };

        \fill[bovert, opacity=0.7] plot coordinates { (0,1)  (-3.2,-0.2) (0,-1) };

        \draw[fill=bovert, opacity=0.7] plot[smooth, tension=0.3] coordinates { (0,-1)  (-1.5,-1)  (-3.2,-0.2) };

\draw [densely dashed, smooth, tension=0.3] plot coordinates {
(-3.2,-0.20)
(-2.8,-0.05)
(-2.4,0.29)
(-2.2,0.17)
(-2.0,0.33)
(-1.8,0.21)
(-1.6,0.36)
(-1.5,0.40) 
(-1.4,0.27)
(-1.2,0.34+0.1)
(-1.0,0.19+0.2)
(-0.8,0.28+0.2)
(-0.6,0.13+0.2)
(-0.4,0.22+0.2)
(-0.2,0.06+0.2)
( 0.0,0.16+0.2)
( 0.2,0.01+0.2)
( 0.4,0.10+0.2)
( 0.6,0.03+0.1)
( 0.8,0.03+0.1)
( 1.0,0.00)
};

        \coordinate (a) at (0,-0.6);
        \coordinate (b) at (0,0.6);
        \coordinate (c) at (0.6,0);
        \draw[very thick,red] (a) -- (b) -- (c) -- cycle;
        \fill[red] (a) circle (1pt)node[left] {$x_k$};
        \fill[red] (b) circle (1pt)node[left] {$z_k$};
        \fill[red] (c) circle (1pt)node[right] {$y_k$};

        \node at (-2.3,0.8) {$\partial_c^+\Omega$};
        \node at (-2.3,-0.9) {$\partial_c^-\Omega$};

    \draw plot[smooth, tension=0.3,] coordinates {
(-3.2,-0.20)
(-2.8,-0.15-0.15)
(-2.4,-0.15-0.1)
(-2.2,-0.17-0.2)
(-2.0,-0.33)
(-1.8,-0.21)
(-1.6,-0.36)
(-1.5,-0.40) 
(-1.4,-0.27)
(-1.2,-0.34-0.1)
(-1.0,-0.19-0.1)
(-0.8,-0.28-0.1)
(-0.6,-0.13-0.1)
(-0.4,-0.22-0.2)
(-0.2,-0.06-0.2)
( 0.0,-0.16-0.2)
( 0.2,-0.01-0.2)
( 0.4,-0.10-0.1)
( 0.6, 0.03-0.1)
( 0.8,-0.03-0.1)
( 1.0, 0.00)
};

\end{tikzpicture}
    \caption{\small{A non-thin family of causal geodesic triangles of $(\Omega,\delta_\Omega)$, inside a bounded causally convex domain $\Omega$ such that $\partial\Omega$ contains a broken lightlike segment (in blue).}}
    \label{fig_non_gromov_hyp_cc}

\end{minipage}
\end{figure}

\subsection{Zooming with similarities}\label{section_zooming} A \emph{similarity} of $\R^{1,n}$ is an affine transformation of $\R^{1,n}$ that is also conformal. Any similarity $g$ can be decomposed as 
$g(x)=\lambda Ax+t,$
where $\lambda>0$ is the \emph{distortion}, $A\in \O(1,n)$ is the \emph{isometric linear part} and $t\in \R^{1,n}$ is the \emph{translation part} of $g$.

\begin{lem}\label{lemme_dynamique_cas_futur_complet}
    Let $\Omega$ be a convex, future complete domain in $\R^{1,n}$ containing no lightlike line. If $\partial\Omega$ is not stably acausal, then there exists a sequence $(g_k)$ of similarities of $\R^{1,n}$ and there exists a convex future complete domain $\Omega^\prime$ containing no lightlike line, such that $g_k(\Omega)\to \Omega^\prime$ and such that $\partial\Omega^\prime$ is not acausal.  
\end{lem}

\begin{proof}
     For simplicity, we can always assume that $\Omega$ is disjoint from the hyperplane $(t=0)$. Assume that $\partial\Omega$ is not stably acausal. Then there exists a sequence $(x_k)$ in $\partial\Omega$, and a sequence of spacelike affine hyperplanes $(H_k)$, such that $H_k$ is a supporting hyperplane of $\Omega$ at $x_k$ for all $k\geq 0$, and such that the sequence $(H_k^\perp)\in \H^n$ converges to a lightlike direction $\Delta\in\partial\H^n$. For $k\geq 0$, write $x_k=(t_k,p_k)$ and let $g_k$ denote the composition of the translation of vector $v_k=(0,-p_k)$ with the dilation of distortion $\lambda_k=1/\max\{1,t_k\}$. Let $k\geq 0$ and let $\Omega_k=g_k(\Omega)$. Then, by construction, the domain $\Omega_k$ is convex, future complete, disjoint from the hyperplane $(t=0)$, and its boundary intersects the lightlike line $(p=0)$ at a point $y_k=g_k(x_k)$ whose $t$-coordinate belongs to $[0,1]$. Also $\Omega_k$ has a supporting hyperplane at $y_k$ of direction $g_k(\vec{H}_k)=\vec{H}_k$. By Proposition \ref{prop_structure_ouverts_causallement_convexes}, we can write $\Omega_k$ as the epigraph of a 1-Lipschitz function $f_k:\R^n\to \R$, for all $k\geq 0$. From what precedes, one has $f_k(0)\in [0,1]$ for all $k\geq 0$. By the Arzel\`a-Ascoli Theorem, up to taking a subsequence, we can always assume that $(f_k)$ converges uniformly on compact subsets to a function $f$. The limit $f$ is still 1-Lipschitz, convex and positive. Hence its epigraph $\Omega^\prime$ is future complete, convex and disjoint from the hyperplane $(t=0)$, and $(\Omega_k)$ converges to $\Omega^ \prime$. Let $y=\lim_k y_k$ and let $H$ be an accumulation point of $g_k(H_k)$. Then $H$ is a supporting hyperplane of $\Omega^\prime$ at $y$, and $H^\perp=\Delta$, so $H$ is lightlike. Since $\Omega^\prime$ is future complete, the lightlike line starting at $y$ and parallel to $\Delta$ is contained in $\partial\Omega^\prime$. Hence $\partial\Omega^\prime$ is not acausal.
\end{proof}

\begin{lem}\label{lemme_dynamique_cas_causalement_convexe}
    Let $\Omega$ be a bounded, convex, causally convex domain in $\R^{1,n}$. If either $\partial_c^+\Omega$ or $\partial_c^-\Omega$ is not stably acausal, then there exists a sequence $(g_k)$ of similarities of $\R^{1,n}$ and there exists a causally convex domain $\Omega^\prime$ containing no lightlike line, such that $g_k(\Omega)\to \Omega^\prime$ and such that $\partial\Omega^\prime$ contains a causal broken lightlike segment.
\end{lem}

\begin{figure}
    \centering
             \begin{tikzpicture}[scale=1.5]

   \begin{scope}[shift={(-4.1,0)}]
        \draw[fill=bovert, opacity=0.7]
  (-1.5,0.6) -- (-1,0.5) -- (-0.5,0.35) -- (0,0) -- (-0.5,-0.15) -- (-1,-0.2) -- (-1.5,-0.2);
   \fill (0,0) circle (1pt)node[right] {$x_k$};
   \draw[thick] (0.5,-0.4)-- (-1,0.8)  node[right]{$H_k$};
   \end{scope}

  \begin{scope}[shift={(-2.8,0.5)}]
        \draw[->,>=Stealth](-0.8,0) to[out=30,in=150](0.8,0) ;
        \node at (0,0.5) {$x\mapsto \lambda_k x$};
  \end{scope}

    \begin{scope}[shift={(1.5,0.3)}]
        \draw[->,>=Stealth](-0.8,0) to[out=30,in=150](0.8,0) ;
        \node at (0,0.8) {  $\begin{pmatrix}
\mu_k &   0 \\
0 & 1/\mu_k
\end{pmatrix}$};
  \end{scope}

      \begin{scope}[shift={(0,0)}]
        \draw[fill=bovert, opacity=0.7]
  (-1.5,0.9) -- (-1,0.75) -- (-0.5,0.4) -- (0,0) -- (-0.5,-0.2) -- (-1,-0.3) -- (-1.5,-0.3);
   \fill (0,0) circle (1pt)node[right] {$x_k$};
   \draw[thick] (0.5,-0.4)-- (-1,0.8)  node[right]{$H_k$};
   \end{scope}

      \begin{scope}[shift={(+4.1,0)}]
        \draw[fill=bovert, opacity=0.7]
  (-1.5,0.9) -- (-1,0.75) -- (-0.5,0.4) -- (0,0) -- (-0.5,-0.4) -- (-1,-0.8) -- (-1.5,-1);
   \fill (0,0) circle (1pt)node[right] {$x_k$};
   \draw[thick] (0.5,-0.4)-- (-1,0.8)  node[right]{$H_k$};
   \end{scope}

\end{tikzpicture}
    \caption{The ``zooming'' step in the proof of Lemma \ref{lemme_dynamique_cas_causalement_convexe}. We depict the effect of the dilation and the isometry part of the similarity $h_k$ on $\Omega$.}
    \label{fig_claim_2}
\end{figure}
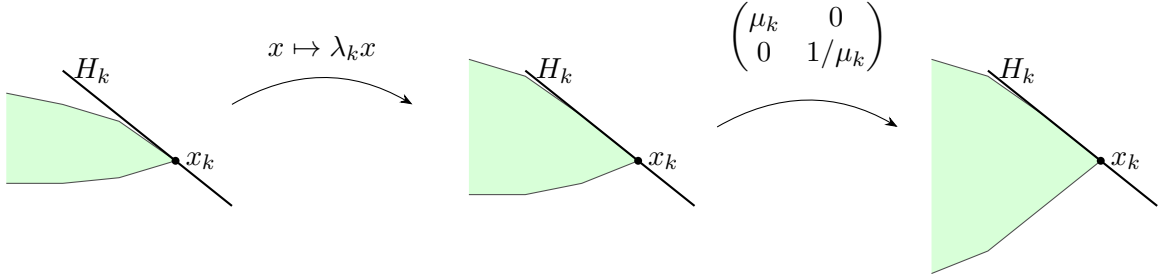

\begin{proof}
    Since $\Omega$ is causally convex, we can write 
    $\Omega=\{(t,p)\in\R\times U\vert\,f_-(p)<t<f_+(p)\},$
    according to Proposition \ref{prop_structure_ouverts_causallement_convexes}. Without loss of generality, assume that $\partial_c^+\Omega$ is not stably acausal. Let $(p_k)$ be a sequence of points in $U$ and let $(H_k)$ be a sequence of hyperplanes such that $H_k$ is a contact hyperplane at $x_k=(f_+(p_k),p_k)$ for all $k\geq 0$, and such that $(H_k)$ converges to a lightlike hyperplane. Fix $v_k\in \vec{H}_k$ such that $(v_k)$ converges to a past lightlike vector $v$, and let $u_k$ be the orthogonal projection of $v_k$ onto $\R^n$, for all $k\geq 0$. We can always replace $p_k$ with the intersection point of $p_k+\R_{>0}v_k$ with $\partial U$ and assume that $p_k\in \partial U$ for all $k\geq 0$.
    
    Let $k\geq 0$, let $(q_m)$ be a sequence of $U$ converging radially to $p_k$ in the direction $u_k$, and let $F_m$ be a sequence of supporting hyperplanes to $\Omega$ at $z_m=(q_m,f_-(q_m))$. Up to taking a subsequence, we can assume that $F_m$ converges to a supporting hyperplane $H^\prime_k$ at $x_k$. The line $\Delta_k=H_k^\prime\cap P_k$ is then asymptotic to $\smash{\overline{\partial_c^-\Omega}}\cap P_k$ at $x_k$. Also, the sequence $(\Delta_k)$ cannot converge to a line directed by $v$, since this would imply that $H_k^\prime\to H $, hence $\Omega$ would be contained in $I^+(H)\cap I^-(H)=\emptyset$, a contradiction. 

    We define the sequence $(g_k)$ of similarities as follows. For every $k\geq 0$, the similarity $g_k$ is centered at $x_k$, has dilation factor $\lambda_k>1$, and its isometry part acts trivially on $P_k^\perp$, and acts by the matrix 
    $$\begin{pmatrix}
\mu_k &   0 \\
0   & 1/\mu_k
\end{pmatrix}$$
on $P_k$, where the previous matrix is written in a basis $(e_k,f_k)$ of lightlike vectors of $P_k$ such that $\lim_k f_k=v$. The quantities $\lambda_k$ and $\mu_k$ are chosen so that $g_k(H_k)\to H$ and such that $g_k(\Delta_k)$ converges to a lightlike direction transverse to $e$. This can be done for a suitable choice of sequence $(\lambda_k)$ and $(\mu_k)$ such that $\lambda_k,\mu_k\to \infty$ and $\mu_k\ll\lambda_k$, because $\lim_k\vec\Delta_k\neq \R v$, see Figure \ref{fig_claim_2}. Up to extracting a subsequence, we may assume that $g_k(H_k^\prime)$ converges to a hyperplane $H^\prime$, which by construction is degenerate. 

Up to extraction, we may assume that the sequence $\Omega_k=g_k(\Omega)$ converges to a convex, causally convex domain $\Omega^\prime$. Since $\Omega^\prime$ is disjoint from $H$ and $H^\prime$, it cannot contain a lightlike line. If $P=\lim_kP_k$, then $P\cap\Omega^\prime$ is a quarter space bounded by a causal broken lightlike line, hence $\partial\Omega^\prime$ contains a causal broken lightlike line.
\end{proof}

\begin{proof}[Proof of Theorem \ref{thm_intro_CNS_cc} and \ref{thm_intro_CNS_futur_complet} (necessary condition)] 
    Let $\Omega$ be a convex, future complete domain containing no lightlike line such that $\partial\Omega$ is not stably acausal, or a bounded, convex and causally convex domain, such that $\partial_c^+\Omega$ or $\partial_c^-\Omega$ is not stably acausal. Combining Lemma \ref{lemme_dynamique_cas_futur_complet}, Lemma \ref{lemme_dynamique_cas_causalement_convexe} and Theorem \ref{thm_coindi_necessaire_faible}, there exists a sequence $(g_k)$ of similarities of $\R^{1,n}$ and there exists a convex, causally convex domain $\Omega^\prime$ containing no lightlike line, such that $g_k(\Omega)\to \Omega^\prime$ and such that $(\Omega^\prime,\delta_{\Omega^\prime})$ is not Gromov hyperbolic. If $(\Omega,\delta_\Omega)$ were $r$-Gromov hyperbolic for some $r>0$, then $g_k(\Omega)$ would be $r$-Gromov hyperbolic for all $k\geq 0$ by Proposition \ref{proposition_naturalité}. This would contradict the fact that $(\Omega^\prime,\delta_{\Omega^\prime})$ is not Gromov hyperbolic by Corollary \ref{cor_limite_de_domaines_Gromov_hyp}. Therefore $(\Omega,\delta_\Omega)$ is not Gromov hyperbolic. 
\end{proof}

\section{Maximal domains are not Gromov hyperbolic}\label{Section_preuve_non_maximalité}

This section is devoted to Theorem \ref{thm_intro_maximalité}. The proof of Theorem \ref{thm_intro_maximalité} is given at the end of this section. We recall the notion of $C$-maximal spacetime introduced in \cite{Clara}. A conformal map $f:M\to N$ between two globally hyperbolic spacetimes is a \emph{Cauchy embedding} if it maps a Cauchy hypersurface of $M$ to a Cauchy hypersurface of $N$. 

\begin{definition}[$C$-maximality]
    A globally hyperbolic conformal spacetime $(M,[g])$ is \emph{$C$-maximal} if every Cauchy embedding $f:M\to N$ is onto.
\end{definition}

Recall that causally convex domains in $\R^{1,n}$ are globally hyperbolic, see Fact~\ref{fait_cc_implique_globalement_hyperbolique}. Causally convex domains in $\R^{1,n}$ that are moreover maximal are well understood.

 \begin{fact}[{see \cite[Prop. 6.3]{Cha_Gal} and \cite[Thm. C]{article_dist_markotitz}}]
    \label{fact_structure_des_CC_maximaux}
     Let $\Omega$ be a domain in $\R^{1,n}$. Then $\Omega$ is causally convex and $C$-maximal
        if and only if $\Omega$ is dually convex in the sense of \cite{Zimpropqh}, that is, for every point $x\in\partial\Omega$, there exists a lightcone $C$ or a lightlike hyperplane $H$ containing $x$ and disjoint from $\Omega$.
    In that case, the metric space $(\Omega,\delta_\Omega)$ is complete as soon as $\Omega$ contains no lightlike line. \qed
 \end{fact}

\begin{lem}
    \label{formule_delta_diamant_dim_1+1}
    Let $D=I(a,b)$ be a diamond in $\R^{1,1}$, and let $e_1,e_2$ be the intersection points of the lightcone of $a$ and $b$. Then, for every $x,y\in D$ that are not causally related, one has 
    $$\delta_D(x,y)=\vert\ln\left(\frac{\b(x-e_1,x-e_1)\b(y-e_2,y-e_2)}{\b(y-e_1,y-e_1)\b(x-e_2,x-e_2)}\right)\vert.$$
    Also, causal curves and achronal curves in $D$ are geodesic for $(D,\delta_D)$.
\end{lem}

\begin{proof}
    This follows from Equation \ref{equation_d_T_diamant} by exchanging the roles of space and time in $\R^{1,1}$.
\end{proof}

\begin{lem}
    \label{Lemme_D1_1}
    Let $D=I(x,y)$ be a diamond in $\R^{1,1}$. Let $a,b\in D$ such that $a\ll b$. Then the restrictions of $\delta_{I(x,y)}$ and $\delta_{I(a,y)}$ to $I(b,y)$ are equivalent. 
\end{lem}

\begin{proof}
    Since $I(a,y)\subset I(x,y)$, we already know from Proposition \ref{proposition_naturalité} that $\delta_{I(x,y)}\leq \delta_{I(a,y)}$ on $I(a,y)$. 
    Let $z\in I(b,y)$. Then, if $\alpha=\ln(\tauL(x,y))/\ln(\tauL(a,y))$, one has
    $$\tauL(x,z)^\alpha\leq \tauL(x,y)^\alpha=\tauL(a,b)\leq \tauL(a,z).$$    
    Let $z_1,z_2\in I(b,y)$. If $z_1\leq z_2$, then from Equation \ref{equation_d_T_diamant}, we deduce
    $$\delta_{I(a,y)}(z_1,z_2)=2\ln\left(\frac{\tauL(z_1,y)\tauL(a,z_2)}{\tauL(z_2,y)\tauL(a,z_1)}\right)\leq 2\ln\left(\frac{\tauL(z_1,y)\tauL(x,z_2)}{\tauL(z_2,y)\tauL(x,z_1)^\alpha}\right)\leq \alpha\delta_{I(x,y)}(z_1,z_2).$$
    If $z_1$ and $z_2$ are arbitrary, let $z_3\in I(b,y)$ be such that $z_1z_3z_2$ is an achronal broken lightlike segment. Then, applying Lemma \ref{formule_delta_diamant_dim_1+1} twice, we get
    $$\begin{aligned}
        \delta_{I(a,y)}(z_1,z_2)&=\delta_{I(a,y)}(z_1,z_3)+\delta_{I(a,y)}(z_3,z_2)\\
        &\leq \alpha\delta_{I(x,y)}(z_1,z_3)+\alpha\delta_{I(x,y)}(z_3,z_2)= \alpha\delta_{I(x,y)}(z_1,z_2).
    \end{aligned}$$
    Hence $\delta_{I(x,y)}\leq \delta_{I(a,y)}\leq \alpha \delta_{I(x,y)}$ on $I(b,y)$.
\end{proof}

\begin{definition}
    Let $\Omega$ be a bounded, causally convex domain in $\R^{1,n}$. A Lorentzian plane $\Sigma$ is called a \emph{future maximal 2-section of $\Omega$} if there exists a connected component $\Omega_\Sigma$ of $\Omega\cap\Sigma$ such that the future causal boundary of $\Omega_\Sigma$ is a broken lightlike segment. See Figure \ref{fig_ccm_bigon}.
\end{definition}

\begin{lem}
    \label{lem_existance_2_sections}
    Let $\Omega$ be a bounded, causally convex, maximal domain in $\R^{1,n}$. Then $\Omega$ admits a future maximal 2-section. 
\end{lem}

\begin{proof}
    Let $x\in \Omega$ and let $U=I^+(x)\cap \Omega$. Then $U$ is conformally equivalent to a convex, past complete regular domain $V$ in $\R^{1,n}$, see \cite{smaï2025futures}. Let $f:U\to V$ be such a conformal equivalence, and let $s_+\in \partial_c^+\Omega$ such that $f(s_+)$ is the terminal singularity of some point of $V$. By \cite[Prop. 4.12]{Bonsante}, there exist two lightlike half-lines in $\partial V$ through $f(s_+)$. In particular, there exist two lightlike segments $c_1$ and $c_2$ through $s_+$ contained in $\partial_c^+\Omega$. Extend these two segments maximally in $\partial_c^+\Omega$, and let $e_1,e_2$ be their respective endpoints. 
    Let $\Sigma$ be the plane containing $e_1,e_2$ and $s_+$ and let $\Omega_\Sigma$ be the connected component of $\Omega\cap\Sigma$ containing $s_+$ in its boundary.  
    Let $y$ be a point in the interior of $c_1$. From Fact \ref{fact_structure_des_CC_maximaux}, there exists $C$, that is either a lightcone or a lightlike hyperplane, such that $C$ contains $y$ and is disjoint from $\Omega$. Assume first that $C$ is a lightcone centered at some point $z\in \R^{1,n}$. If $z\not\in \Sigma$, then $C\cap \Sigma$ would be a spacelike curve disconnecting $\Omega_\Sigma$, which would be a contradiction. Hence $z\in \Sigma$ and $C\cap \Sigma$ is a union of two lightlike lines, one of which contains $c_1$. Assume now that $C$ is a degenerate hyperplane. Then, arguing similarly, the intersection $C\cap \Sigma$ is a lightlike line parallel to $c_1$. It follows from both cases that $\Omega_\Sigma$ is disjoint from the line generated by $c_1$. Similarly  $\Omega_\Sigma$ is disjoint from the line generated by $c_2$. Hence $\Omega_\Sigma$ is contained in the past of $s_+$. It follows from the maximality of $c_1$ and $c_2$,  that $\partial_c^+\Omega_\Sigma=c_1\cup c_2$.  
\end{proof}

\begin{lem}
    \label{lemme_2_quasi_geod_CCM}
    Let $\Omega$ be a bounded, causally convex domain in $\R^{1,n}$, and let $\Sigma$ be a future maximal 2-section of $\Omega$. Let $\Omega_\Sigma$ be a connected component of $\Omega\cap \Sigma$ whose future causal boundary is a broken lightlike segment. Then, for every $a\in \Omega_\Sigma$, there exists $A>0$ such that every acausal curve in $\Omega_\Sigma\cap I^+(a)$ is a $(A,0)$-quasi-geodesic of $(\Omega,\delta_\Omega)$. 
\end{lem}

\begin{proof}
    Let $e_1$ and $e_2$ be the endpoints of $\partial_c^+\Omega_\Sigma$. Note that $\Omega$ cannot intersect $I^\pm(e_1)$ and $I^\pm(e_2)$. Indeed, assume for instance that $x\in \Omega\cap I^-(e_1)$. The diamond $I(x,e_1)$ is contained in $\Omega$, since $\Omega$ is causally convex. Then $I(x,e_1)\cap \Sigma$ is non-empty and contained in $\Omega\cap \Sigma\cap I^-(e_1)$. This contradicts the fact that $e_1\not \in \partial_c^+\Omega_\Sigma$. Let $U$ be the set of points in $\R^{1,n}$ that are not causally related to $e_1$ or $e_2$. Let $f:U\to \R_{>0}$ be the map defined by $f(x)=\frac{\b(x-e_1,x-e_1)}{\b(x-e_2,x-e_2)}$. Then $f$ is projective, hence by Proposition \ref{Proposition_applications_projectives}, one has 
    $$\vert\ln(f(x)/f(y))\vert\leq \delta_U(x,y),$$
    for every $x,y\in U$.
    We denote by $s_-$ the intersection point of $\Sigma$ with the past lightcones of $e_1$ and $e_2$. We write $D=I(s_-,s_+)\cap\Sigma$. Fix a point $b\in\Omega\cap I^-(a)$ and write $S=\Sigma\cap I(b,s_+)$. Let $x,y\in \Omega_\Sigma\cap I^+(a)$ be two non-causally related points. The inclusions $S\subset\Omega\subset U$ imply, by Proposition \ref{proposition_naturalité}, that 
    $$ \delta_U(x,y)\leq\delta_\Omega(x,y)\leq \delta_S(x,y).$$
    From Lemma \ref{Lemme_D1_1}, we can find a constant $A \geq 1$ such that $\delta_S\leq A\delta_D$ on $\Omega_\Sigma\cap I^+(a)$. Now, since $x$ and $y$ are not causally related, Lemma \ref{formule_delta_diamant_dim_1+1} implies that
    $$\delta_{D}(x,y)=\vert\ln(f(x)/f(y))\vert.$$    
    Combining the preceding inequalities, we obtain
    $$\delta_D(x,y)\leq \delta_\Omega(x,y)\leq A\delta_D(x,y).$$
    Since achronal curves in $D$ are geodesic for $\delta_D$ by Lemma \ref{formule_delta_diamant_dim_1+1}, any achronal curve of $\Omega_\Sigma\cap I^+(a)$ is a $(A,0)$-quasi-geodesic of $(\Omega,\delta_\Omega)$.
\end{proof}

\begin{figure}
    \centering
        \begin{tabular}{ccccc}
    \begin{tikzpicture}[scale=0.9]

    \coordinate (sp) at (-0.50*1.801,1.8);
    \coordinate (eun) at (-1.3*1.801,-0.45);
    \coordinate (edeux) at (0*1.801,0.6);

    \draw[thick,opacity=0.7] 
    (0.3*1.801,3.5) to[out=-100, in=110] 
    (0.6*1.801,-2) to[out=-120, in=-30] 
    (-1.9*1.801,-3) to[out=80, in=-80] 
    (-1.9*1.801,2.5) to[out=-10, in=-130] 
    (0.3*1.801,3.5);
    \fill[fill=white]  (sp) -- (1.2*1.801,1)
    -- (1.2*1.801,-1.6) -- (0*1.801,-1.8) ;

    \fill[fill=bovert]  (sp) -- 
    (eun) -- 
    (-1.1*1.801,-0.6) --
    (-0.8*1.801,-1.15) --
    (-0.5*1.801,-1.46) -- 
    (-0.3*1.801,-0.9) --
    (-0.1*1.801,-0.5) -- (edeux) ;

    \draw[thick] (eun) -- 
    (-1.1*1.801,-0.6) --
    (-0.8*1.801,-1.15) --
    (-0.5*1.801,-1.46) ;

    \draw[thick, densely dashed] (-0.5*1.801,-1.46) -- 
    (-0.3*1.801,-0.9) --
    (-0.1*1.801,-0.5) -- (edeux) ;

    \fill (sp) circle (2pt) node [above right] {$s_+$};
    \fill (eun) circle (2pt) node [above left, xshift=7pt, yshift=2] {$e_1$};
    \fill (edeux) circle (2pt) node [above right,  xshift=-5pt] {$e_2$};
    
    \draw[thick, densely dashed] (edeux)
    -- (1*1.801,-0.1)
    -- (2*1.801,-0.5);
    \draw[thick, densely dashed, opacity=0.5] (-1.9*1.801,0.75)
    -- (-1.5*1.801,1) 
    -- (edeux);
    \draw[thick]  (-1.75*1.801,0.04)
     -- (-2*1.801,0.5);

     \draw[thick]  (2*1.801,-0.5)
    -- (1.9*1.801,-0.75)
    -- (1.5*1.801,-1)
    -- (0*1.801,-0.8) 
    -- (-0.6*1.801,-0.6)
    -- (eun);
     \draw[thick, opacity = 0.5]  (eun)
     -- (-2*1.801,0.5)
     -- (-1.9*1.801,0.75);
     \draw[thick]  (-1.75*1.801,0.17)
     -- (-2*1.801,0.5)
     -- (-1.9*1.801,0.75);

     \draw[thick] (-1.77*1.801,0.9)
     -- (-1.9*1.801,0.75);
     
    \draw[thick,opacity=0.4]  (-2*1.801,0.5) 
        -- (-1.2*1.801,-1) 
        -- (-0.5*1.801,-1.46);
    \draw[thick] (-0.5*1.801,-1.46)
     -- (0*1.801,-1.8) 
     -- (1.2*1.801,-1.6)     
    -- (1.9*1.801,-0.75) ;

    \draw[thick]  (2*1.801,-0.5) 
        -- (1.2*1.801,1) 
     -- (sp)  ;
    \draw[thick,opacity=0.5] (sp) 
     -- (-1.2*1.801,1.6)     
    -- (-1.9*1.801,0.75) ;
    \draw[thick]  (sp) -- (eun) ;
    \draw[thick, densely dashed]  (sp) -- (edeux) ;

    \draw[fill=blue,opacity=0.05] 
    (-1.9*1.801,2.5) to[out=-10, in=-130] 
    (0.3*1.801,3.5) to[out=-100, in=110] 
    (0.6*1.801,-2) to[out=-120, in=-30] 
    (-1.9*1.801,-3) to[out=80, in=-80] 
    (-1.9*1.801,2.5);
    \draw[thick,opacity=0.4] 
    (0.3*1.801,3.5) to[out=-100, in=110] 
    (0.6*1.801,-2) to[out=-120, in=-30] 
    (-1.9*1.801,-3) to[out=80, in=-80] 
    (-1.9*1.801,2.5) to[out=-10, in=-130] 
    (0.3*1.801,3.5);

    \node at (-1.7*1.801,2) {$\Sigma$};
    \node at (1.7*1.801,0.7) {$\Omega$};
        
\end{tikzpicture}
             
             &  &&&

    \begin{tikzpicture}
    \def\a{2.6}
    \def\B{2.6}
    \draw (\a,\B) -- (-\a,\B) -- (-\a,-\B) -- (\a,-\B) -- cycle;
    \draw[fill=blue,opacity=0.03] (\a,\B) -- (-\a,\B) -- (-\a,-\B) -- (\a,-\B) -- cycle;
    \draw[fill=bovert, opacity=1]  (0,2) -- (-2,0) 
    -- (-1.5,-0.3) 
    -- (-1,-0.7) 
    -- (-0.5,-0.6) 
    -- (0,-0.8) 
    -- (0.5,-0.8) 
    -- (1,-0.4) 
    -- (1.5,-0.3) 
    -- (2,0) -- cycle;
    \draw (0,2) -- (-2,0) -- (0,-2) -- (2,0) -- cycle;
    \draw (0,2) -- (-1,1) -- (0,0) -- (1,1) -- cycle;

    \coordinate (a) at (-0.6,1);
    \coordinate (b) at (0,1.6);
    \coordinate (c) at (0.6,1);
    \coordinate (d) at (0,0.4);
    \draw[very thick,red] (a) -- (b) -- (c)  -- cycle;
    \fill[red] (a) circle (2pt);
    \fill[red] (b) circle (2pt);
    \fill[red] (c) circle (2pt);

    \fill (0,0) circle (2pt) node [left] {$a$};
    \fill (-2,0) circle (2pt) node [left] {$e_1$};
    \fill (2,0) circle (2pt) node [right] {$e_2$};
    \fill (0,2) circle (2pt) node [above left] {$s_+$};
    \fill (0,-2) circle (2pt) node [below left] {$s_-$};
    \node at (-2.3,2.2) {$\Sigma$};
    \node at (1.3,0) {$\Omega_\Sigma$};
        
\end{tikzpicture}
        \end{tabular}
    \caption{On the left: a future maximal 2-section $\Sigma$ of a causally convex maximal domain $\Omega$. On the right, in red: an achronal quasi-geodesic triangle in $\Omega_\Sigma$.}
    \label{fig_ccm_bigon}
\end{figure}
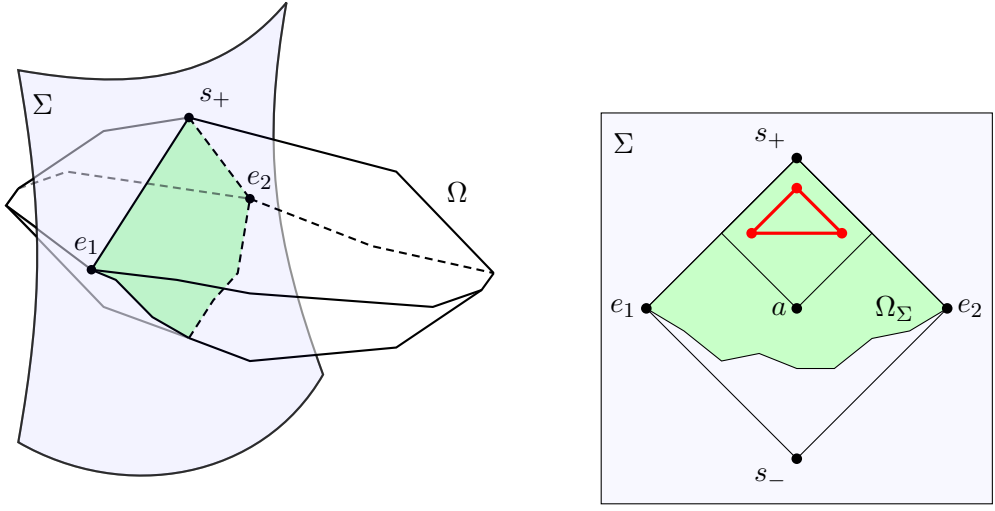

\begin{proof}[Proof of Theorem \ref{thm_intro_maximalité}]
     Let $\Omega$ be a bounded, causally convex, maximal domain in $\R^{1,n}$, let $\Sigma$ be a future maximal 2-section of $\Omega$, let $\Omega_\Sigma$ be a connected component of $\Omega\cap \Sigma$ whose future causal boundary is a broken lightlike segment, let $a\in \Omega_\Sigma$ and let $A>0$ be the constant given by Lemma \ref{lemme_2_quasi_geod_CCM}.  Let $(x_ky_kz_k)$ be a sequence of achronal triangles inside $\Omega_\Sigma\cap I^+(a)$ such that $[x_k,y_k]$ and $[y_k,z_k]$ are lightlike segments, for all $k\geq 0$. One can choose $(x_ky_kz_k)$ so that $[x_k,y_k]$ and $[y_k,z_k]$ converge to lightlike segments in the future causal boundary of $\Omega_\Sigma$, and such that $[x_k,z_k]$ converges to a spacelike segment, see Figure \ref{fig_ccm_bigon}. In particular, the edges $[x_k,y_k]$ and $[y_k,z_k]$ leave every compact subset of $\Omega$, but $[x_k,z_k]$ intersect a fixed compact subset of $\Omega$. By Lemma \ref{lemme_2_quasi_geod_CCM}, the triangle $x_ky_kz_k$ is a $(A,0)$-quasi-geodesic triangle for all $k\geq 0$, and by Fact \ref{fact_structure_des_CC_maximaux}, the metric space $(\Omega,\delta_\Omega)$ is complete. We deduce from Proposition \ref{lemme_triangle_non_fins} that $(\Omega,\delta_\Omega)$ is not Gromov hyperbolic. 
\end{proof}

\section{Examples}\label{section_examples} 

\subsection{Homogeneous Bonsante domains}\label{section_HBD} A homogeneous Bonsante domain is a domain $\Omega_\ell$ that can be expressed as the future $I^+(F)$ of a $\ell$-dimensional Riemannian subset $F\subset \R^{1,n}$. 
Explicitly, one can write 
    $$\Omega_\ell=\left\{(t,p)\in\R^{1,n}\,\vert\, t>\sqrt{p_1^2+\dots+p_{n-\ell}^2}\right\},$$
in a suitable basis. These domains are convex and future complete, hence satisfy the hypothesis of Theorem \ref{thm_intro_CNS_futur_complet}. The author proved the following proposition in \cite[Prop. 11.2]{article_dist_markotitz}.

\begin{prop}
    \label{prop_equivalence_d_Omega_ell}
    There exists a bi-Lipschitz mapping between $(\Omega_\ell,\delta_{\Omega_\ell})$ and $\H^{n-\ell}\times \H^{\ell+1}$.\qed
\end{prop}

It follows from Proposition \ref{prop_equivalence_d_Omega_ell} that $(\Omega_\ell,\delta_{\Omega_\ell})$ is Gromov hyperbolic if and only if $\ell=n$. This is in accordance with Theorem \ref{thm_intro_CNS_futur_complet}, since $\partial\Omega_\ell$ is not acausal when $\ell<n$, and $\partial\Omega_\ell$ is $\varepsilon$-stably acausal for every $\varepsilon>0$ when $\ell=n$. 
Note that the quasi-hyperbolic metric $\hk_{\Omega_\ell}=t^{-2}h$ is simply the hyperbolic metric of the upper half-space when $\ell=n$, and $F_{\Omega_\ell}=\frac{1}{\sqrt2}\hk_{\Omega_\ell}$ in that case. From Corollary \ref{cor_general_metric_conf_plat_et_d_light}, one has 
$$\frac{1}{\sqrt 2}\K_{\Omega_n}\leq \delta_{\Omega_n}\leq \K_{\Omega_n}.$$
Compare also with \cite[Thm. 5]{markowitz_warped_product}.

\subsection{The spacelike slab}
    \label{ex_spacelike_slab}
    Let $\Omega=\{(t,p)\in\R^{1,n}\,\vert\,\vert t\vert<1\}$ be a spacelike slab. Then $\Omega$ is causally convex and $\partial_c^+\Omega$ and $\partial_c^-\Omega$ are stably acausal. From Theorem \ref{thm_causally_thin}, we know that $(\Omega,\delta_\Omega)$ is causally thin. However, the metric space $(\Omega,\delta_\Omega)$ is not Gromov hyperbolic. To see this, it is sufficient to show that $(\Omega,\K_\Omega)$ is not Gromov hyperbolic by Theorem \ref{thm_equiv_explicite_delta_Omega_et_qh}. Now the hyperplane $(t=0)$ is a totally geodesic copy of $\R^n$ inside $(\Omega,\K_\Omega)$. Hence $(\Omega,\K_\Omega)$ contains Euclidean triangles and $(\Omega,\K_\Omega)$ is not Gromov hyperbolic. 
    Note also that $\delta_\Omega$ coincides with the null distance of the time function  $\tau(t,x)=\ln((1+t)/(1-t))$ by Lemma \ref{lemme_temps_projectif}, since $(t,x)\to t$ is a projective Cauchy time function.

\section{Comparison with the Hilbert metric}\label{Section_hilbert}

Let $\Omega$ be a convex domain in $\R^{1,n}$ containing no complete affine line, and let $H_\Omega$ denote the Hilbert metric on $\Omega$. Recall that the Hilbert metric is obtained by integrating the Finsler metric 
$$G_\Omega(v)=\frac{\vert v \vert}{d_h(x,x_v^-)}+\frac{\vert v \vert}{d_h(x,x_v^+)},$$
along piecewise smooth curves in $\Omega$, where $x_v^\pm$ are the endpoints of the geodesic segment generated by $v\in T_x\Omega$ in $\Omega$.
It is clear from Proposition \ref{formule_F_Omega} that  $H_\Omega\leq \delta_\Omega$ holds on $\Omega$, but it is not clear in general if these two distances are equivalent, that is, if there exists a constant $\alpha\geq 1$ such that 
$\delta_\Omega\leq \alpha H_\Omega$, see for instance \cite[Ex. 7.8]{article_dist_markotitz} for the case of the disc.

\begin{cor}[of Theorem \ref{thm_intro_CNS_cc}]
    \label{corollaire_pour_hilbert}
    There exist infinitely many bounded convex domains $\Omega$ in $\R^{1,n}$ such that $\delta_\Omega$ and $H_\Omega$
    are not quasi-isometrically equivalent.
\end{cor}

\begin{proof}
    Let $\Omega$ be a bounded, convex and causally convex domain in $\R^{1,n}$ with stably acausal boundary. From Theorem \ref{thm_intro_CNS_cc}, the metric space $(\Omega,\delta_\Omega)$ is Gromov hyperbolic. Now the boundary of $\Omega$ is not $C^1$ at points of $\partial(\partial_c^+\Omega)$. Hence $(\Omega,H_\Omega)$ is not Gromov hyperbolic, see \cite{karlsson_Noskov,Benoist_convexes_quasisymetriques}. In particular, $H_\Omega$ and $\delta_\Omega$ are not quasi-isometrically equivalent by Proposition \ref{prop_dist_quasi_iso}. 
\end{proof}

In dimension $n+1=2$, there exist, up to a similarity, exactly four domains such that the Markowitz metric is equivalent to the Hilbert metric.

\begin{figure}
    \centering
         
\begin{tikzpicture}[scale=1]

\def\h{0.2} 
\def\colorNS{red!50} 
\def\colorOE{blue!50} 

\def\SmoothConvex{
  (0.2,0) (3,0.3) (3.2,1.8) (2.2,3) (1.2,3.2) (0,1.5)
}

\filldraw[blue!10, thick]
  plot [smooth cycle, tension=0.7] coordinates {\SmoothConvex};

\begin{scope}
  \clip plot [smooth cycle, tension=0.7] coordinates {\SmoothConvex};

  \foreach \x in {0,\h,...,3.5} {
    \draw[\colorNS, thin] (\x,-0.5) -- (\x,3.5);
  }

  \foreach \y in {0,\h,...,3.5} {
    \draw[\colorOE, thin] (-0.5,\y) -- (3.5,\y);
  }
\end{scope}

\draw[thick] (1,0.6) -- (1.2,0.6) -- (1.2,1.2) -- (1.6,1.2) -- (1.6,1.4) -- (1.8,1.4) -- (1.8,1.6) -- (2.2,1.6) -- (2.2,2);

\fill[black] (1,0.6) circle (1.5pt) node [left] {$x$};
\fill[black] (2.2,2) circle (1.5pt) node [right] {$y$};

\draw[thick]
  plot [smooth cycle, tension=0.7] coordinates {\SmoothConvex};

\end{tikzpicture}

    \caption{The Markowitz metric of a convex subset of $\R^{1,1}=(\R^2,[dxdy])$ can be thought of as a ``Manhattan-Hilbert metric''. It is the metric one would obtain in a convex city, where the street grid is of Manhattan type, but where distances are measured according to the Hilbert metric rather than the Euclidean one.}
    \label{Figure_manhattan_convex}
\end{figure}
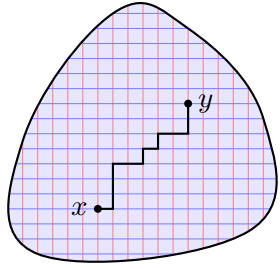

\begin{thm}
    Let $\Omega$ be a convex domain in $\R^{1,1}$, containing no complete line. Then the following propositions are equivalent:
    \begin{enumerate}
        \item $\delta_\Omega$ and $H_\Omega$ are equivalent;
        \item $\delta_\Omega$ and $H_\Omega$ are quasi-isometrically equivalent;        
        \item $\Omega$ is Möbius equivalent to a diamond;
        \item $\Omega$ is a diamond, a space component, a future, a past, or a half lightlike slab.
    \end{enumerate}
\end{thm}

In the previous statement, a \emph{space component} refers to a connected component of $\R^{1,1}\setminus J^\pm(x)$, for some $x\in \R^{1,1}$. A \emph{half lightlike slab} corresponds to the intersection of a lightlike slab with the future or the past of a point. A \emph{Möbius transformation} is the composition of similarities with the Lorentzian inversion $i(x)=\frac{x}{\b(x,x)}$.

\begin{proof}
    We first prove $(2)\Rightarrow (4)$. Assume that there exist constants $\alpha,\beta>0$ such that $\delta_\Omega\leq \alpha H_\Omega+\beta$. Let $x\in \partial \Omega$ be such that $\partial\Omega$ is differentiable at $x$. Let $\Delta$ be the tangent line to $\Omega$ at $x$, and let $U$ be the half-space bounded by $\Delta$ and containing $\Omega$. We claim that $\Delta$ is lightlike. Assume instead that $\Delta$ is nondegenerate. Without loss of generality, we can assume that $\Delta$ is Riemannian. For $k\geq 0$, let $h_k$ be the dilation at $x$ of distortion $2^k$, and let $\Omega_k=h_k(\Omega)$. Then $(\Omega_k)$ converges to $U$ in the Hausdorff topology. Let $a,b\in U$, and let $k$ be large enough, so that $a,b\in \Omega_k$. Then 
    $$\delta_U(a,b)\leq \delta_{\Omega_k}(a,b)\leq \alpha H_{\Omega_k}(a,b)+\beta.$$
    Choose $a$ and $b$ so that $[a,b]$ is collinear to $\Delta$. Then $H_{\Omega_k}(a,b)\to 0$ as $k\to \infty$, hence $\delta_U(a,b)\leq \beta$. This contradicts the fact that $\delta_U(a,b)$ can be chosen arbitrarily large. Hence $\Delta$ is lightlike. It follows that the boundary of $\Omega$ is a union of lightlike segments or half-lines, and that $\Omega$ is a diamond, a space component, a future, a past or a half lightlike slab. 
    
    We now prove $(4)\Rightarrow (1)$. We distinguish three cases.
    \begin{enumerate}
        \item \emph{$\Omega$ is a future, a past or a space component}. Then there exists $\alpha>0$ such that $\delta_\Omega\leq \alpha H_\Omega$ holds on a neighborhood of a point of $\Omega$. Since the group $\O(1,1)$ acts by similarities of $\Omega$, its action is isometric with respect to both $\delta_\Omega$ and $H_\Omega$. Since this action is transitive, and since $\delta_\Omega$ and $H_\Omega$ are both length metrics, we deduce that $\delta_\Omega\leq \alpha H_\Omega$ holds on $\Omega$, see for instance \cite[Cor. 4.9]{article_dist_markotitz}. Note that the constant $\alpha>0$ is independent of $\Omega$.
        \item \emph{$\Omega$ is a half lightlike slab}. Up to applying a similarity, we can assume that 
        $$\Omega=\{(t,p)\in \R^{1,1} \,\vert\, t+p>0\text{ and }\vert t-p\vert<1\}.$$ 
        Let $\Delta=\{t=p\}$ and let $U=\Omega\cap J^-(\Delta)$ and $V=\Omega\cap J^+(\Delta)$ be the two subdomains of $\Omega$ bounded by $\Delta$. Let $I=I^+((-1/2,1/2))$. 
        Let $x\in U$ and $v\in T_xU$ be lightlike. If $v$ is collinear to $(1,1)$, then $F_\Omega(v)=F_U(v)$. If not, then 
        $$F_\Omega(v)=\frac{\vert v\vert}{d(x,x_v^+)}+\frac{\vert v\vert}{d(x,x_v^-)}\leq 2\frac{\vert v\vert}{d(x,x_v^-)}=2F_I(v).$$
        From the previous case, we deduce that 
        $$\delta_\Omega\leq 2\delta_I\leq 2\alpha H_I\leq 2\alpha H_\Omega,\text{ on $U$}.$$
        Similarly, one has $\delta_\Omega\leq  2\alpha H_\Omega$ on $V$. Since $\delta_\Omega$ and $H_\Omega$ are length metrics on $\Omega$, it follows that $\delta_\Omega\leq  2\alpha H_\Omega$ on $\Omega$. 
        \item \emph{$\Omega$ is a diamond.} Up to applying a similarity, we can assume that 
        $$\Omega=\{(t,p)\in \R^{1,1} \,\vert\, \vert t\pm p\vert<1\}.$$ 
        Let $D_1,\dots,D_4$ be the four diamonds of $\Omega$ bounded by the lightcone of the origin. Then, as in the previous case, one has $\delta_\Omega\leq  2\alpha H_\Omega$ on $D_i$, for every $i=1,\dots,4$. We deduce that $\delta_\Omega\leq  2\alpha H_\Omega$ on $\Omega$.
    \end{enumerate}
    In every case, we have $\delta_\Omega\leq  2\alpha H_\Omega$, hence $\delta_\Omega$ and $H_\Omega$ are equivalent. Implication $(1)\Rightarrow (2)$ is immediate. The equivalence $(4)\Leftrightarrow (3)$ is standard and holds in any dimension, see for instance \cite[Sect. 1.4.5]{thèse}.
\end{proof}

\subsection*{Declaration} This research was funded in part by the Luxembourg National Research Fund (FNR), grant reference O24/18936913/RiGA.  For the purpose of open access, and in fulfillment of the obligations arising from the grant agreement, the author has applied a Creative Commons Attribution 4.0 International (CC BY 4.0) license to any Author Accepted Manuscript version arising from this submission.

\bibliographystyle{amsalpha} 
\bibliography{biblio}

\end{document}